\documentclass[11pt,english]{article}

\usepackage[utf8x]{inputenc}
\usepackage[T1]{fontenc}
\usepackage{cmlgc}
\usepackage{ucs}

\usepackage{xcolor}
\usepackage[left=2.3cm, right=2.3cm,top=2.5cm]{geometry}
\usepackage{graphicx}
\usepackage{mathtools,amsfonts,amssymb, amsthm}
\usepackage{mathrsfs} 
\usepackage[all]{xy}

\definecolor{vertfonce}{rgb}{0.20, 0.46, 0.25}
\definecolor{rougefonce}{rgb}{0.64, 0.09, 0.20}

\usepackage{babel}
\usepackage[breaklinks=true,
            colorlinks=true,
            linkcolor=rougefonce,
            citecolor=vertfonce]{hyperref}

\usepackage{autonum}
\input{macro}

\title{On the symplectic geometry of $A_k$ singularities}

\author{Nikolay N. \textsc{Martynchuk}\footnote{University of Groningen, Bernoulli Institute for Mathematics, Computer Science and Artificial Intelligence, P.O. Box 407, 9700 AK Groningen, The Netherlands,
email: n.martynchuk@rug.nl} \and \textsc{V\~{u}
    Ng\d{o}c}
  San\footnote{Univ Rennes, CNRS, IRMAR - UMR 6625, F-35000 Rennes,
    France, email: san.vu-ngoc@univ-rennes1.fr}}

\date{}

\renewcommand{\O}{\mathcal{O}}

\theoremstyle{definition}
\newtheorem{remark}[theo]{Remark}

\begin{document}
\maketitle

{\bf Abstract} This paper presents a complete symplectic
classification of $A_k$ Hamiltonians on $\RM^2$, in the analytic and
smooth categories. Precisely, consider the pair $(H, \omega)$
consisting of a Hamiltonian and a symplectic structure on
$\mathbb R^2$ such that $H$ has an $A_{k-1}$ singularity at the origin
with $k\geq 2$.  We classify such pairs near the origin, up to
fiberwise symplectomorphisms, and up to $H$-preserving
symplectomorphisms. The classification is obtained by bringing the
pair $(H, \omega)$ to a symplectic normal form
$$\big(H = \xi^2 \pm x^k, \ \omega = \textup{d} (f \textup{d} \xi)\big), \quad f =
\sum_{i=1}^{k-1} x^i f_i(x^k),$$ modulo some relations which are
explicitly given. 
 We also
show that the group of $H$-preserving symplectomorphisms of an
$A_{k-1}$ singularity with $k$ odd consists of symplectomorphisms
that can be included into a $C^\infty$-smooth (resp., real-analytic)
$H$-preserving flow, whereas for $k$ even with $k \ge 4$  the same is true modulo the
$\mathbb Z_2$-subgroup generated by the involution $\textup{Inv}(x,\xi) = (-x,-\xi)$.
The paper is concluded with a brief discussion of the conjecture that the symplectic
invariants of $A_{k-1}$ singularities are spectrally determined.

\begin{footnotesize}
  \noindent \textbf{MS Classification :}
  37J39, 
  70H06, 
  37J35, 
  70H15, 
  58J50 
\end{footnotesize}

\section{Introduction}

Consider a one degree of freedom Hamiltonian system $(H, \omega)$ on
$\mathbb R^2$, that is, a smooth function $H$ and a symplectic
structure $\omega$ on $\mathbb R^2$. In what follows and unless stated
otherwise, we shall assume that the Hamiltonian $H$ has an isolated
singularity at the origin of type $A_{k-1}$; this means that in some
(generally non-canonical) local coordinates around the origin, the
Hamiltonian $H$ can be written as
$$
H(x,\xi) = \xi^2 \pm x^k + \textup{const};
$$
in what follows, $\textup{const} = 0.$ See Fig~\ref{fibrations_Ak}.
\begin{figure}[htbp]
  \begin{center}
    \includegraphics[width=0.87\linewidth]{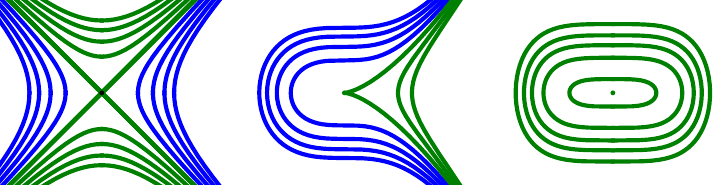}
    \caption{Level sets of $A_{k}, \ k = 1,2,3,$ singularities:
      $H = \xi^2 - x^{2}$ (left), $H = \xi^2 - x^{3}$ (middle),
      $H = \xi^2 + x^{4}$ (right).}
    \label{fibrations_Ak}
  \end{center}
\end{figure}
In this paper, we give a classification of such pairs
$(H = \xi^2 \pm x^k, \omega)$ under local transformations preserving
either the Hamiltonian $H$ or the corresponding fibration $\{H = h\}$,
in the $C^\infty$-smooth and analytic categories. More precisely, the
group $\textup{Diff}(\RM^2,0)$ of local ($C^\infty$-smooth,
respectively analytic) diffeomorphism germs at the origin acts on
pairs $(H,\omega)$ by taking the pull-back:
\[
  (\psi^* H, \psi^* \omega), \quad \psi \in \textup{Diff}(\RM^2,0)\,.
\]
 In the $H$-preserving case, we consider
two pairs $(H_0,\omega_0)$ and $(H_1,\omega_1)$ to be equivalent if
they belong to the same orbit of $\textup{Diff}(\RM^2,0)$, and in the
fibration preserving case we additionally allow transformations on the
left, that is, we consider two pairs $(H_0,\omega_0)$ and
$(H_1,\omega_1)$ to be equivalent if there exists $\tilde H_1$ that
defines the same (singular) foliation as $H_1$ and $(H_0,\omega_0)$
and $(\tilde H_1,\omega_1)$ are in the same orbit of $\textup{Diff}(\RM^2,0)$.

The main interest in this classification comes from the fact that
$A_k, \ k > 1,$ singularities represent one of the simplest examples
of degenerate singularities of integrable Hamiltonian systems and in the case
$k = 2\ell -1, \ \ell > 1,$ (and $H = \xi^2 + x^{2\ell}$), the action
variable and the corresponding integer-affine structure do not suffice
for their symplectic classification. Specifically, this means that
there exist germs of symplectic structures $\omega_0$ and $\omega_1$
near the origin such that the pairs
$ (H = \xi^2 + x^{2\ell}, \omega_0)$ and
$(H = \xi^2 + x^{2\ell}, \omega_1)$ with $\ell \ge 2$ have the same
actions: for all $h > 0$ sufficiently small,
$$
I_0(h) = \int_{\{H \le h\}} \omega_0 = \int_{\{H \le h\}} \omega_1 =
I_1(h),
$$
while there exists no diffeomorphism germ preserving the fibration
$\{H = h\}$ and sending $\omega_1$ to $\omega_0$\footnote{At least in
  the analytic setting, such counterexamples are known; see in
  particular the PhD thesis \cite{Guglielmi2018}.  We will supply a
  proof later in the paper; see
  Section~\ref{sec/classification_smooth}.}. This observation is in
striking contrast with essentially all of the local singularities for
which the symplectic classification is currently known, including the
$A_1$ singularity \cite{DeVerdiere1979}, more general elliptic and
focus-focus singularities \cite{Eliasson1984, Eliasson1990, Miranda2004,
  Vu-Ngoc2013}, and also the $A_2$ singularity~\cite{Bolsinov2018,
  Kudryavtseva2021}, all of which typically occur in 2 or 3 degree of
freedom integrable systems, see for
instance~\cite{vanderMeer1985,CdV2003,Waalkens2004,Bolsinov2004,Dhont2008,Efstathiou2012,Cushman2015,Bolsinov2018_2,Hohloch2021}. This
is also in contrast with most of the known global classifications
provided the singularities are simple and non-splittable
\cite{Delzant1988, Pelayo2009, Pelayo2011, Kudryavtseva2020,
  Kudryavtseva2022}.  See
also the works ~\cite{Duistermaat1980,Zung1996_2,Vu-Ngoc2003,Bolsinov2004,Garay2005,Vu-Ngoc2006,Lukina2008,Izosimov2011b,Cushman2015,CdV2003,Bolsinov2018_2, 
Bolsinov2019,Floch2019,Alonso2020} for
general classification problems of integrable systems.

One particular implication of the insufficiency of the action variable
is that it is currently unknown whether the symplectic invariants of
the Hamiltonian system
$$(H = \xi^2 + x^{2\ell} + O(\|(x,\xi)\|^{2\ell+1}), \textup{d}\xi \wedge \textup{d}x), \ \ell \ge 2,$$
near the origin are spectrally determined: the usual (second order)
Bohr-Sommerfeld rules allow one to reconstruct only the action
variables from the spectrum of the Weyl quantisation $\widehat H$ of $H$;
this suffices in the elliptic case \cite{Vu-Ngoc2011}, but already not
so in the quartic situation (see in particular the discussion in Section~\ref{sec/discussion}).

All this naturally gives rise to the questions of
\begin{itemize}
\item[(i)] what extra symplectic invariants (apart from the action
  variable) exist?
\item[(ii)] what are the symplectic normal forms (up to
  transformations preserving $H$ or, resp., the fibration induced by
  $H$)?
\item[(iii)] does the spectrum of the Weyl quantisation $\widehat H$ still
  contain all the information about the symplectic invariants?
\end{itemize}
In this work, we present a complete solution to the first two of these
questions. In particular, we show that $(H = \xi^2 \pm x^k, \omega)$
can be brought to a normal form
$$\big(H =  \xi^2 \pm x^k, 
\omega = \sum_{i=0}^{k-2} x^i
c_i(H)\big),$$ where $c_i$ is unique up to certain (explicitly given) relations, which
depend on the type of classification ($H$-preserving or fibration-preserving), 
the parity of $k$, and the sign in the Hamiltonian
$H = \xi^2 \pm x^k$ as well as the smoothness class:
$C^\infty$-smooth or real-analytic.
The functions $c_i$ taken up 
to these relations form a complete and minimal set of symplectic invariants,
in the sense that two  pairs  $(H =  \xi^2 \pm x^k, 
\omega_0)$ and $(H =  \xi^2 \pm x^k, 
\omega_1)$ have the same invariants if and only if they are equivalent (in their respective regularity class) up to
\begin{itemize}
\item[a)]
$H$-preserving diffeomorphisms, that is,
are in the same orbit of $\textup{Diff}(\RM^2,0)$, or, respectively, 
\item[b)] fibration-preserving diffeomorphisms, that is, there exists
  a Hamiltonian $\tilde H$ defining the same (singular) foliation as
  $H$ such that $(H, \omega_0)$ and $(\tilde H, \omega_1)$ are in the
  same orbit of $\textup{Diff}(\RM^2,0)$. \end{itemize} We note that
in the complex-analytic category, these results are essentially known
\cite{Francoise1978, Francoise1988}. What was missing is a real
symplectic classification, especially in the $C^\infty$-smooth
category, and the present paper is meant to fill in this gap; it also
presents a new formulation of the classification, see
Theorems~\ref{theo:classical-analytic} and
~\ref{theo:classical-smooth} below, while providing a more explicit
construction of the invariants.

For example, for $k = 4$ and
$H = \xi^2 + x^4$, the germs of $c_{0}$ and $c_2$ and the Taylor
series of $\pm c_{1}$ at the origin (i.e., taken up to sign\footnote{That the Taylor
  series of $c_1$ needs to be taken up to sign comes from the map
  $\textup{Inv}(x,\xi) = (-x,-\xi)$ which is an $H$-preserving
  diffeomorphism that does not change $c_0$ and $c_2$, but changes
  $c_1$.})  classify the symplectic structures up to $H$-preserving
$C^\infty$-smooth diffeomorphisms.  
The germs $c_0 = c_0(H)$ and
$c_2 = c_2(H)$ correspond to the action variable, and in the fibration
preserving case, $c_0$ can be set to one. The Taylor series of $c_1$
(taken up to sign) is an additional symplectic invariant that is not
captured by the action variable.  See
Sections~\ref{sec/classification_analytic} and
\ref{sec/classification_smooth} for more detail.

One application of our results states that for $A_{k-1}$ singularities,
there exists a canonical change of coordinates such that the
Hamiltonian function still splits in the two variables at the cost of
modifying the potential function:

\begin{theo}
  \label{theo:potential}
  Let $H \colon \RM^2 \to \RM$ have an $A_{k-1}$ singularity at the
  origin. Then there exists a smooth (in the real-analytic case,
  real-analytic) local diffeomorphism $\phy:(\RM^2,0)\to(\RM^2,0)$
  preserving the canonical symplectic structure on $\mathbb R^2$ such
  that
  \[
    H\circ \phy = \tilde \xi^2 + V(\tilde x),
  \]
  for some smooth (in the real-analytic case, real-analytic) function
  $V$.
\end{theo}
Interestingly, this shows that the familiar energy from classical
mechanics (kinetic plus potential energy) turns out to be a universal
model for all 1D Hamiltonians near an $A_{k-1}$ singularity.

\begin{proof}[Proof of Theorem~\ref{theo:potential}]
  The proof follows from Theorem~\ref{theo:classical-smooth} below. Indeed, according to this theorem,
  there exist (generally non-canonical) coordinates $(x,\xi)$ on $\mathbb R^2$ near the origin such
  that
  \[
    H = \xi^2 \pm x^{k}, \quad k \ge 2,
  \]
  and the symplectic structure has the form
  $
  \omega = \textup{d}(f \textup{d} \xi)
  $
  for some $C^\infty$-smooth function $f = f(x)$ with $f'(0) \ne 0$. Consider the change of
  variables $\tilde \xi = - \xi, \ \tilde x = f(x)$. Then 
    \[
    H(\tilde x, \tilde \xi) = \tilde \xi^2 + V(\tilde x),
  \]
  where $V = \pm (f^{-1}(\tilde x))^{k},$ and $\omega = \textup{d} \tilde \xi \wedge \textup{d} \tilde x.$
\end{proof}

Finally, we note that the symplectic normal forms constructed in this from are, in fact, explicit and amount to the following 
straightforward procedure:

\begin{itemize}
\item[a)] Finding (generally non-canonical) coordinates $(x, \xi)$ for which the Hamiltonian $H$ of an $A_{k-1}$ singularity
has the form $H = \xi^2 \pm x^k$;

\item[b)] Solving a cohomological equation (as in \cite{Francoise1978, Francoise1988} in the analytic case or by solving a finite number of integral equations;
see Section~\ref{sec/classification_smooth}) and

\item[c)] applying the Moser path method.

\end{itemize}

All of this steps can be achieved by virtue of explicit formulas (see
in particular \cite{Bolsinov2018} for an explicit formula for the
Moser path method in the autonomous case). This may be of interest in
relation to the study of important models from classical mechanics,
see for instance the work~\cite{Francoise2013}, where explicit normal
forms are given near relative equilibria of the Euler top.

The paper is organised as follows. In Section~\ref{sec/moser}, we
reduce the problem of symplectic classification to its infinitesimal
counterpart, by determining which energy-preserving symplectomorphisms
arise via the Moser path method. In the next
Sections~\ref{sec/classification_analytic} and
\ref{sec/classification_smooth}, we classify the $A_{k}$ singularities
up to energy-preserving and fibration-preserving symplectic
transformations.  The paper is concluded with a brief discussion of the
conjecture that all of the symplectic invariants of $A_k$
singularities are spectrally determined.

\section{$H$-preserving isotopies and symplectic
  classification} \label{sec/moser}

Let $\omega_0$ and $\omega_1, \ \omega_1/\omega_0 > 0,$ be two
symplectic forms defined in a neighborhood of the origin $O$ in
$\mathbb R^2$.  Assume that $H$ has an isolated singularity at $O$,
such as the $A_k$ singularity. The question that we will now address
is: when does there exist a local diffeomorphism $\Phi$ such that
$\Phi^{*}\omega_1 = \omega_0$ and $H \circ \Phi = H$? 
(We will use this later also for
symplectic classification up to fibration-preserving diffeomorphisms.)

\subsection{Moser isotopies}

A sufficient condition for the existence of such a diffeomorphism
$\Phi$ is given by the following lemma, which is an adaptation of the
classical Moser's trick, and works both in the $C^\infty$ and analytic
categories. Let $(\phy_t)_{t\in[0,1]}$ be an isotopy of
diffeomorphisms. (In this paper, we always impose
$\phy_0=\textup{Id}$). We say that this isotopy is $H$-preserving when
$\phy_t^* H = H$ for all $t\in [0,1]$.

\begin{lemm} \label{lemm:moser} \textup{(See \cite{Francoise1978}.)}
Suppose that there exists a smooth
  solution $u$ of the cohomological equation
  \begin{equation} \label{eq/Moser_trick} \textup{d}H \wedge
    \textup{d}u = \omega_1 - \omega_0.
  \end{equation}
  Then there exists a local $H$-preserving isotopy of diffeomorphisms
  $\Phi_t$ such that $\Phi_0 = \textup{Id}$ and
  $\Phi_1^{*}\omega_1 = \omega_0$.
\end{lemm}
\begin{proof}
 The proof follows \cite[Eq.~(5.1) and (5.2)]{Bolsinov2018} (see also \cite{Kudryavtseva2021}) and gives an explicit construction of required isotopy. 
 Let $\ham{H}$ denote the Hamiltonian vector field of $H$
  with respect to $\omega_0$, that is, $i_{\ham{H}} \omega_0 = - \textup{d} H.$ 
The required isotopy $\Phi_t$ can then simply be defined as the inverse of the time $\tau(t,x) = t u(x)$
map of the flow of $\ham{H}$.
Indeed, let $\Psi_t$ be the time $\tau(t,x) = t u(x)$
map of $\ham{H}$, so that $\Phi_t = \Psi_t^{-1}$.
Then
  \[
   \Psi_0 = \textup{id}, \quad \frac{\textup{d}}{\textup{d}t} \Psi_{t}(x) = u(x) \ham{H}(\Psi_{t}(x))
  \]
 and, letting $X_t =  u(\Psi^{-1}_{t}(x)) \ham{H}(x)$,
  \[
    \frac{\textup{d}}{\textup{d}t}\Psi_{t}^{*} \omega_0 = 
  \Psi_{t}^{*}(\textup{d} i_{X_t} \omega_0) = \Psi_{t}^{*}(\dd H \wedge \dd (u \circ \Psi^{-1}_{t})) = \Psi_{t}^{*}((\Psi^{-1}_{t})^{*}\dd H \wedge (\Psi^{-1}_{t})^{*}\dd u) = \omega_1 - \omega_0.
  \]
  Therefore,
  $\Psi_t^*\omega_0 = \omega_0 + t (\omega_1 - \omega_0).$ Setting $t = 1$, we get 
    $\Psi_{1}^* \omega_0 = \omega_1$ and hence 
    $$\Phi_{1}^*\omega_1 = (\Psi_1^{-1})^* \omega_1 =  \omega_0.$$
Note that since $\Psi_t$ is constructed via a multiple of the Hamiltonian vector field $\ham{H}$
  for which
  $\dd H(\ham{H})=0$, $\Psi_t$ and $\Phi_t$ leave $H$ invariant. It
  follows that $\Phi_t$ is as required.
\end{proof}

\begin{remark}
The explicit construction of an isotopy given in the above proof is of interest
in connection to the work \cite{Francoise2013}, where explicit normal forms are given near relative equilibria of the Euler top:
the other steps in the procedure of obtaining a symplectic normal for $A_k$ singularities that we give in this paper are (also) explicit.
\end{remark}

An $H$-preserving isotopy of diffeomorphisms obtained by the above
construction will be called ``of Moser type''. The
requirement~\eqref{eq/Moser_trick} has a natural cohomological
interpretation, as follows. Since in dimension 2, $\Omega^2=Z^2$
(\emph{i.e.} all $2$-forms are closed), we have by the Poincar{\'e}
lemma the following short exact sequence (on $\RM^2$ or any ball
around the origin)
\[\xymatrix{ \mathbb R \ar[r]& \Omega^0 \ar[r]^{\dd{}}& \Omega^1
    \ar[r]^{\dd{}}& Z^2 \ar[r]& 0}\,,
\]
giving the isomorphism $\delta:\frac{\Omega^1}{\dd{\Omega^0}}\to
Z^2$. The Moser method can be applied if we find a smooth $u$ such
that $\omega_1-\omega_0 = \dd H \wedge \dd u$, which is equivalent via
$\delta^{-1}$ to $\alpha_0 - \alpha_1 = u \dd H + \dd f$, for some
smooth function $f$, where $\alpha_j$ is a primitive of $\omega_j$; in
other words, if
$\delta^{-1}(\omega_1-\omega_0) \in \Omega^0 \dd H +
\dd{\Omega_0}$. Therefore, we naturally introduce the ``deformation
space''
\[
  \mathcal{H}_{H} := \frac{\Omega^1}{\Omega^0 \dd H + \dd{\Omega_0}}
  \simeq \frac{Z^2}{\dd{\Omega^0}\wedge \dd H}
\]
(sometimes called the vanishing cohomology, see~\cite{CdV2003}), and
denote by $[\omega_1-\omega_0]_H$ the class of
$\delta^{-1}(\omega_1-\omega_0)$ in $\mathcal{H} _H$.
Condition~\eqref{eq/Moser_trick} is $[\omega_1-\omega_0]_H = 0$. It
turns out that, in the case of an $A_k$ singularity, there is a
converse statement.
\begin{prop} \label{prop/character_moser_type} Let $\omega_0$,
  $\omega_1$ be symplectic forms near $0\in\RM^2$. Let
  $H(x,\xi) = \xi^2 \pm x^k$, $k\geq 2$. Then, there exists an
  $H$-preserving isotopy of diffeomorphisms $\Psi_t$ such that
  $\Psi_0=\textup{Id}$ and $\Psi_1^*\omega_1 = \omega_0$ if and only
  if $[\omega_1-\omega_0]_H = 0$. In this case, there exists an
  isotopy of Moser type with the same property.
\end{prop}
\begin{proof}
  The last sentence is merely Lemma~\ref{lemm:moser}; hence we need to
  prove that the existence of the isotopy $\Psi_t$ (not necessarily of
  Moser type) implies the cohomological equation
  $[\omega_1-\omega_0]_H = 0$. Let
  $\omega_t = (\Psi_t^{-1})^*\omega_0$ and let $\ham{H}^t$ be minus
  the Hamiltonian vector field of $H$ for the $\omega_t$ symplectic
  form, \emph{i.e.}
  \begin{equation}
    i_{\ham{H}^t}\omega_t = \dd H\,.
    \label{equ:hamHt}
  \end{equation}
  Let $X_t$ be the vector field of the isotopy, that is
  \[
    X_t(\Psi_t) = \frac{\dd{\Psi_t}}{\dd{t}}\,.
  \]
  Since $\Psi_t$ preserves $H$ we must have $\dd H(X_t) = 0$, which
  from~\eqref{equ:hamHt} gives $\omega_t(\ham{H}^t,X_t)=0$. Hence
  $X_t$ is symplectically orthogonal to $\ham{H}^t$ which implies, by
  dimension consideration, that these vector fields must be collinear:
  there is a function $u_t$ such that
  $$X_t = u_t \ham{H}^t.$$ Let us
  prove that $u_t$ must be smooth.  Indeed, we have that $X_t$ is
  smooth as a function of $(x,\xi)$ for every $t$.  Let
  $\omega_t = g(x,\xi,t) \dd \xi \wedge \dd x.$ Then
  $\ham{H}^t = (-2\xi, \pm k x^{k-1})/g$. Hence
  $(-2\xi u_t, \pm k x^{k-1} u_t) = X_tg$ and therefore also
  $(\xi u_t, x^{k-1} u_t)$ are smooth vector fields. Suppose $u_t$ is
  not smooth. Then the Taylor expansion of $\xi u_t$ at the origin (we
  only need to check smoothness at the origin) contains a non-zero
  term of the form $c x^m$. But then the Taylor expansion of
  $x^{k-1} u_t = (x^{k-1}/ \xi) \xi u_t$ contains a non-zero term
  $c x^{m+k-1}/ \xi,$ which implies $c = 0$.  This is a contradiction
  showing $u_t$ is smooth as a function of $(x,\xi)$ for every $t$.

  Since $\Psi_t^* \omega_t = \omega_0$, we have
  \begin{align}
    \label{eq:3}
    0 = \frac{\dd{}}{\dd t}\Psi_t^* \omega_t
    &= \Psi_t^* \left( \lie_{X_t}\omega_t + \frac{\dd{\omega_t}}{\dd t}\right)\\
    &= \Psi_t^*\left( \dd (i_{X_t}\omega_t) +  \frac{\dd{\omega_t}}{\dd t}\right)\\
    &= \Psi_t^*\left(\dd u_t \wedge \dd H + \frac{\dd{\omega_t}}{\dd t}\right)
  \end{align}
  and hence
  \[
    \dd H \wedge \dd u_t = \frac{\dd{\omega_t}}{\dd t}.
  \]
  Integrating over $[0,1]$, we see that the function
  $u:=\int_0^1 u_t \dd t$ satisfies our cohomological equation
  \[
    \dd H \wedge \dd u = \omega_1 - \omega_0\,, \quad \text{
      \emph{i.e.} } \quad [\omega_1-\omega_0]_H = 0\,.
  \]
  The result follows (since by Lemma~\ref{lemm:moser}, there is an
  $H$-preserving isotopy $\Phi_t$ of Moser type such that
  $\Phi_1^* \omega_1 = \omega_0$).
\end{proof}

\subsection{$H$-preserving diffeomorphisms vs $H$-preserving
  isotopies}

Let $H(x,\xi) = \xi^2 \pm x^k$, with $k\geq 3$.  Before we proceed to
the symplectic classification, we need to know which $H$-preserving
diffeomorphisms of an $A_k$ singularity arise as $H$-preserving
isotopies. The following result is in contrast with the usual Morse
case $k=2$.

\begin{lemm}\label{lemm:A}
  Let $H(x,\xi) = \xi^2 \pm x^k$, with $k\geq 3,$ and $\phy$ be a
  $C^\infty$-smooth (respectively, real-analytic) diffeomorphism
  fixing the origin and such that $\phy^* H = H$.  Then there exists
  $\epsilon_1=\pm 1$, $\epsilon_2=\pm 1$ (with $\epsilon_1^k=1$), and
  $b\in\RM$ such that the linear part of $\phy,$ in coordinates
  $(x, \xi)$, is
  \[
    \dd \phy(0) =
    \begin{pmatrix}
      \epsilon_1&b\\0&\epsilon_2
    \end{pmatrix} .
  \]
  Let $A:=\begin{pmatrix} \epsilon_1&0\\0&\epsilon_2
  \end{pmatrix}$. Then, the map $\tilde \phy := \phy\circ A$ preserves
  $H$ and is homotopic to the identity through a $C^\infty$-smooth (in
  the real-analytic case, real-analytic) $H$-preserving isotopy.
\end{lemm}

\begin{proof}
  We denote $\phy(x,\xi) = (\phy_x(x,\xi), \, \phy_\xi(x,\xi))$, and
  $\dd \phy(0)=:
  \begin{pmatrix}
    a&b\\c&d
  \end{pmatrix}$. We denote by $\O(k)$ the ideal of smooth functions
  on $\RM^2$ whose Taylor series at the origin starts with a
  homogeneous polynomial in $(x,\xi)$ of degree $\geq k$. Writing
  \begin{equation}
    \phy_\xi^2 \pm \phy_x^k = \xi^2 \pm x^k,
    \label{equ:preserveH}
  \end{equation}
  and using that $\phy_x^k \in \O(k)$ we get
  \begin{equation}
    \phy_\xi^2 = \xi^2 + \O(k).
    \label{equ:phy_xi2}
  \end{equation}
  Since $\phy_\xi(x,\xi)^2 = (cx + d\xi)^2 + \O(3)$ and $k\geq 3$, we
  obtain, identifying coefficients, that $c=0$ and $d^2=1$. Hence
  $\phy_\xi(x,\xi) = d \xi + \O(2)$, $d=\pm 1$.  The full Taylor
  expansion of $\phy_\xi$ is
  \[
    \phy_\xi(x,\xi) = d \xi + A_2 + A_3 + \cdots
  \]
  where $A_j$ is homogeneous of degree $j$. Let $j_0\geq 2$ be the
  smallest index for which $A_{j_0}\neq 0$. Then the term of smallest
  degree in $\phy_\xi^2 - \xi^2$ is $2d\xi A_{j_0}$, of degree
  $j_0+1$. From~\eqref{equ:phy_xi2} we deduce that $j_0+1 \geq k$.
  Hence we have
  \begin{equation}
    \label{equ:phy_xi}
    \phy_\xi(x,\xi) = d \xi + \O(k-1), \quad d = \pm 1,
  \end{equation}
  and $\phy_\xi^2 = \xi^2 + 2d \xi A_{k-1} + \O(k+1) $. Writing now
  \begin{equation}
    \phy_x(x,\xi) = ax + b\xi + S_2(x,\xi),
    \label{equ:phy_x}
  \end{equation}
  where $S_2\in \O(2)$, we have
  \[
    \phy_x(x,\xi)^k = (ax + b\xi)^k + \O(k+1) = (ax)^k + \xi D_{k-1} +
    \O(k+1)
  \]
  where $D_{k-1}$ is homogeneous of degree
  $k-1$. From~\eqref{equ:preserveH}, we obtain
  \[
    \pm (ax)^k \pm \xi D_{k-1} + \xi^2 + 2d \xi A_{k-1} = \pm x^k +
    \xi^2 + \O(k+1),
  \]
  which implies $a^k=1$, and $\pm D_{k-1} + 2d A_{k-1} = 0$.

  For each $t>0$, consider the maps $h_t(x,\xi) = (t^2x, t^{k}\xi)$,
  and
  \begin{equation} \label{isotopy} \phy_t(x,\xi) =
    \left(\frac{1}{t^2}\phy_x(h_t(x,\xi)), \,
      \frac{1}{t^{k}}\phy_\xi(h_t(x,\xi))\right)\,.
  \end{equation}
  
  Observe that, from~\eqref{equ:phy_xi} and~\eqref{equ:phy_x},
  \[
    \phy_t(x,\xi) = \left(ax + bt^{k-2}\xi + \frac{1}{t^2}
      S_2(t^2x,t^{k}\xi), \, d \xi + \frac{1}{t^{k}} B_{k-1}(t^2x,
      t^{k}\xi)\right)\,,
  \]
  where $B_{k-1}\in\mathcal{O}(k-1)$.  It follows that $\phy_t$ is
  smooth with respect to all variables (including the case $t = 0$)
  and converges to the linear map $\phy_0(x,\xi) = (ax, d \xi)$ as
  $t \to 0$.

  To conclude, we let $\epsilon_1:= a$, $\epsilon_2 := d$, and we see
  that $\phy_t\circ A$ is the required isotopy. Indeed, we have
  already shown that $\phy_t$ and hence $\phy_t\circ A$ is
  smooth. Furthermore, $\phy_0\circ A$ is the identity. Finally,
  notice that
  \begin{align}
    \phy_t^* H
    &=       \frac{1}{t^{2k}}\phy_\xi^2\circ h_t \pm 
      \frac{1}{t^{2k}}\phy_x^k\circ h_t  \\
    &= \frac{1}{t^{2k}} (\phy^*H)\circ h_t\\
    &= \frac{1}{t^{2k}} H\circ h_t = H\,.
  \end{align}
  This shows that the isotopy preserves $H$.
\end{proof}

\begin{remark}
  In case when $\varphi$ preserves the `canonical' symplectic
  structure $\dd \xi \wedge \dd x$, then so does the isotopy
  \eqref{isotopy}:
  $\varphi_t^*(\dd \xi \wedge \dd x) = \dd \xi \wedge \dd x$ for all
  $t \in [0,1]$.
\end{remark}

The following theorem will not be directly used in what follows, but
can be proven by the same method.
\begin{theo}
  \label{theo:symplecto}
  Let $\omega$ be a symplectic structure around the origin in
  $\mathbb R^2$ and $\varphi$ be a $C^\infty$-smooth (respectively,
  real-analytic) diffeomorphism fixing the origin such that
  $\varphi^*\omega = \omega$ and $H \circ \varphi = H$, where
  $H = \xi^2 \pm x^k, \ k > 2.$ Moreover, let
  $\textup{Inv} \colon \mathbb R^2 \to \mathbb R^2$ be defined by
  $\textup{Inv} (x,\xi) = (-x, -\xi).$ If $k = 2\ell +1$ is odd, then
  $\varphi$ must be a time-one map of the flow of a function of
  $H$. If $k = 2\ell$ is even either $\varphi$ or
  $\varphi \circ \textup{Inv}$ is a time-one map of the flow of a
  function that is constant on the connected components of $H.$
\end{theo}
\begin{remark}
  In the case when $H = \xi^2 + x^{2\ell}$, then $\varphi$ or
  $\varphi \circ \textup{Inv}$ is a time-one map of the flow of a
  \textit{function} of $H,$ since the fibers $H^{-1}(h)$ are connected
  in this case. In the real-analytic case, $\varphi$ or
  $\varphi \circ \textup{Inv}$ is a time-one map of the flow of a
  function of $H$ also when $H = \xi^2 - x^{2\ell}.$
\end{remark}
\begin{proof}[Proof of Theorem~\ref{theo:symplecto}]
  By Lemma~\ref{lemm:A}, we can without loss of generality assume that
  $\varphi$ is isotopic to the identity via a $C^\infty$-smooth
  (respectively, real-analytic) $H$-preserving isotopy.  Let $\phy_t$
  be such an isotopy and $X_t$ its vector field, given by
  \[
    X_t(\phy_t) = \frac{\dd{\phy_t}}{\dd{t}}\,.
  \]
  As in the proof of Proposition~\ref{prop/character_moser_type}, we
  have that $X_t = u_t \ham{H}$ for a $C^\infty$-smooth (in the
  real-analytic case, real-analytic) function $u_t$, where $\ham{H}$
  is minus the Hamiltonian vector field of $H.$ Moreover, the function
  $u:=\int_0^1 u_t(\varphi_t)$ satisfies the cohomological equation
  $ \dd H \wedge \dd u = 0\,$ since $\varphi = \varphi_1$ preserves
  $\omega$. In particular, $u$ is constant on the connected components
  of $H$.  Then any smooth function $g$ such that $\dd g = u \dd H$ is
  as required.  Indeed, let $B_\varepsilon \subset \mathbb R^2$ be a
  small open ball around the origin. Then on each connected component
  of $B_\varepsilon \setminus \{H = 0\}$ we can write
$$
\omega = \textup{d} H \wedge \textup{d} V
$$
for some variable $V$, which is multivalued when
$H = \xi^2+ x^{2\ell}$ and single-valued otherwise. (This is a simple
version of the Darboux-Carath{\'e}odory theorem, which in this case
directly follows from $\dd H \ne 0$ outside $H = 0$.)  Observe that in
the $(V, H)$ coordinates,
$$
\frac{\dd{\phy_t}}{\dd{t}} = u_t \ham{H}|_{\varphi_t} =
\left(u_t(\varphi_t), 0\right)
$$
and hence
$$
\varphi(1) - \varphi(0) = \int_0^1 \left(u_t(\varphi_t), 0\right)
\textup{d}t = (u, 0).
$$
But the same is true for the time-one map of the flow of
$\ham{g}$. The result follows.
\end{proof}

\subsection{Symplectic classification} \label{subsec/classification}

The preceding discussion shows that in order to reduce the pair
$(H,\omega_0)$ to a new, hopefully simpler pair $(H,\omega_1)$, it
essentially suffices to apply Lemma~\ref{lemm:moser}, \emph{i.e.}
find a smooth function $u$ solving the cohomological equation
\eqref{eq/Moser_trick}.  Indeed, for any $H$-preserving
symplectomorphism $\varphi$ with
$\varphi^* \omega_1 = \omega_0, \, \omega_1/\omega_0 > 0,$ it holds
that either $\varphi$ or $\varphi \circ \textup{Inv},$ where
$\textup{Inv}(x,\xi) = (-x,-\xi),$ arises from an $H$-preserving
isotopy of Moser type; see Proposition~\ref{prop/character_moser_type}
and Lemma~\ref{lemm:A}.

Observe that Eq.~\eqref{eq/Moser_trick} is simply a linear PDE (a
transport equation) on the unknown function $u$, for it can be
rewritten as
\begin{equation} \label{eq/PDE} \partial_x u \cdot \partial_\xi H -
  \partial_\xi u \cdot \partial_x H= g(x, \xi),
\end{equation}
where
$$
\omega_1 - \omega_0 = g(x, \xi) \textup{d}\xi \wedge \textup{d}x.
$$
We note that Eq.~\eqref{eq/PDE} can always be formally solved, as
shown below, but not always can one find a smooth (or even an
everywhere well-defined) solution. Indeed, let us look for a solution
of the form $u = \tilde u(\xi,H(x,\xi))$. We have
$\dd H \wedge \dd u = -\partial_\xi\tilde u \partial_x H \dd \xi
\wedge \dd x$. Hence, let $\varepsilon = \varepsilon(x,\xi)$ be a
function that is constant on the connected components of the level
sets of $H$\footnote{We will take $\varepsilon = \textup{const}$ when
  $H = \xi^2 \pm x^{2\ell + 1}$ or $H = \xi^2 - x^{2\ell}$ and
  $\varepsilon = \sqrt{h}$ when $H = \xi^2 + x^{2\ell} = h.$}. Then
the following formula, coming from the initial condition
$\tilde u |_{x = -\varepsilon(x,\xi)}=0$,
\begin{equation}
  \label{eq/solution_cohomolog} \tilde{u}(\xi,h) =
  \int_{-\varepsilon(x,\xi)}^{\xi} \dfrac{g(\beta(t,h),t)}{-\partial_
    x H (\beta(t,h),t)} \textup{d}t, \quad \text{ where } \ \
  H(\beta(t,h),t) = t^2 \pm \beta^k(t,h) = h,
\end{equation}
gives rise to a solution $u = \tilde{u}(\xi,H(x,\xi))$ of
Eq.~\eqref{eq/PDE}; the general solution is then
$\tilde{u}(\xi,H(x,\xi)) + a(x,\xi)$, where $a = a(x,\xi)$ is constant
on the connected components of the levels sets of $H$.

We will use these observations in the following way. We start with
local coordinates $(x,\xi)$ around the origin in $\mathbb R^2$ in
which the Hamiltonian
$$H(x,\xi) = \xi^2 \pm x^k,$$
but the symplectic structure
$\omega_0 = g_0(x,\xi) \textup{d}\xi \wedge \textup{d}x$ is arbitrary.
Then we will bring the symplectic structure to the form
$\omega_0 = \omega_{norm} + \textup{d}H \wedge \textup{d}u$, where
$\omega_{norm}$ is in normal form giving (possibly modulo the action
of the involution $\textup{Inv}(x,\xi) = (-x,-\xi)$) the desired
symplectic invariant of the $H$-preserving symplectic equivalence.
Finally, the normal form $\omega_{norm}$ will be simplified further by
allowing `transformations on the left', resulting in the normal form
of the fibration $H \colon \mathbb R^2 \to \mathbb R$ near the origin.

The idea behind the procedure goes back to \cite{DeVerdiere1979}; see
also \cite{Francoise1978, Francoise1988, Francoise1990} and the more recent works
\cite{Bolsinov2018} and \cite{Kudryavtseva2021} very close to this
point of view. The procedure is general and can in theory be applied
to any Hamiltonian function.  In this paper, we restrict our attention
to Hamiltonians having an isolated singularity of the type $A_k$,
which is a simplest (possibly) degenerate singularity in one
dimension.

Note that so far the discussion in the smooth and real-analytic
categories has been completely parallel. The precise results and the
proofs however will be somewhat different in these two cases. In what
follows, the two cases are hence considered separately.

\section{Symplectic classification in the real-analytic
  case} \label{sec/classification_analytic}

Let $H$ be a real-analytic Hamiltonian on $(\RM^2, \omega)$
with an $A_{k-1}, k \ge 2,$ singularity at the origin $O$. The following result gives a local
symplectic classification of $(H, \omega)$ near $O$ up to $H$-preserving real-analytic diffeomorphism germs, as defined in the Introduction.

\begin{theo}
  \label{theo:classical-analytic}
Let $H$ be a real-analytic Hamiltonian on $(\RM^2, \omega)$
having $A_{k-1}, k \ge 2,$ singularity at the origin $O$.
  Then the pair $(H, \omega)$ has the following local symplectic normal form near $O$ for the group of
  $H$-preserving real-analytic diffeomorphism germs:
  \begin{equation}
    \big(H = \xi^2 \pm x^{k}, \ \omega = \textup{d}(f \textup{d} \xi)\big), \quad \mbox{ where } f =
    \sum_{i=1}^{k-1} x^i f_i(x^k),\label{equ:normal-form_analytic}
  \end{equation}
  for some real-analytic functions $f_i$ with $f_1(0) \ne
  0$ defined uniquely modulo the following relation: if $k = 2\ell$ is even, then  $f$ is
    uniquely defined up to changing the sign of
    $\sum_{i=1}^{\ell-1} x^{2i} f_{2i}(x^{2\ell})$. In other words,
      \begin{itemize}
  \item if $k$ is odd, then (the Taylor series of) $f_{i}, i = 1, \ldots, k -1,$ form a complete
  and minimal set of symplectic invariants;
  \item if $k = 2\ell$ is even, then (the Taylor series of) $f_{i}, i = 1, \ldots, k -1,$ form a complete
    set of symplectic invariants with the only relation given by changing the sign of
    $\sum_{i=1}^{\ell-1} x^{2i} f_{2i}(x^{2\ell})$.
  \end{itemize}
\end{theo}
\begin{remark}
The ambiguity in the sign of
  $\sum_{i=1}^{\ell-1} x^{2i} f_i(x^{2\ell})$ comes from the following fact. If $k$ is even,
  the
  transformation $\textup{Inv}(x,\xi) = (-x,-\xi)$ is an
  $H$-preserving diffeomorphism that transforms the function $f(x)$ from 
  Eq.~\eqref{equ:normal-form_analytic} to $-f(-x)$. However, when $k = 2\ell$ with $\ell \ge 2$, 
  $\textup{Inv}(x,\xi)$ cannot be included into a
  smooth $H$-preserving flow and hence is not of Moser type: when
  $H = \xi^2 - x^{2\ell}$ this follows since $\textup{Inv}$ swaps the
  connected components of the level sets of $H$; when $H = \xi^2 + x^{2\ell},$ one can show that the
  `action variable'
  $$
  I'(h)/2 = \int_{-\sqrt{h}}^{\sqrt{h}} \frac{\dd \xi}{k (\xi^2 -
    h)^{\frac{k-1}{k}}} = h^{\frac{1-\ell}{2\ell}} \int_{-1}^{1} \frac{\dd s}{k (s^2 -
    1)^{\frac{k-1}{k}}}
  $$
  being singular for $\ell > 1$ is an obstruction; see \cite{Kudryavtseva2021b}.
  
  In the odd case, on the contrary, every $H$-preserving and
  orientation-preserving diffeomorphism is of Moser type by
  Proposition~\ref{prop/character_moser_type} and Lemma~\ref{lemm:A}.
\end{remark}
In order to prove Theorem~\ref{theo:classical-analytic}, let us first
show that one can achieve the normal form~\eqref{equ:normal-form_analytic} by
means of Moser isotopies, which, by Lemma~\ref{lemm:moser} amounts to
solving the cohomological equation
\[
  \dd H \wedge \dd u = (g(x, \xi) + f'(x))\dd \xi \wedge \dd x\,,
\]
for some smooth functions $u$ and $f$. In order to simplify notation,
we identify $\Omega^2$ with $\Omega^0$ using the canonical symplectic
form $\dd \xi \wedge \dd x$, and the corresponding Poisson bracket
$\{\cdot,\cdot\}$. Hence the above equation now reads:
\[
  \{H, u\} = g(x, \xi) + f'(x)\,.
\]

\begin{prop}\label{prop:formal}
  Let $H = \xi^2 \pm x^k$, $k\geq 2$.  For any formal series
  $g(x,\xi)\in\RM\formel{x,\xi}$ there exists a unique formal series
  of the form $c(x) = \sum_{i=0}^{k-2}x^i c_i(x^k)$,
  $c_i\in\RM\formel{x}$, for which the equation
  \begin{equation}
    \{H, u\} = g(x, \xi) - c(x)\label{equ:formal}  
  \end{equation}
  admits a formal solution $u(x,\xi)$. Moreover, if $g$ is analytic at
  the origin $(0,0)\in\RM^2$, then one can choose $u$ to be analytic
  at the origin, and hence all $c_i$'s are analytic at $x=0$.
\end{prop}
\begin{remark}
  It will follow from the proof that the radius of convergence of $u$
  and $c_i$ is at least $R/2$, where $R$ is the radius of convergence
  of $g$, see~\eqref{equ:radius}.
\end{remark}
\begin{proof}[Proof of Proposition~\ref{prop:formal}]~\\
  \noindent\textbf{Step 1.} First observe that without loss of
  generality, $g = g(x, \xi)$ is an even function of $\xi$. Indeed,
  let $g(x, \xi) = g_0(x, \xi^2) + \xi g_1(x, \xi^2)$. Then the
  function $\check u = \check u(x,\xi)$ given by
  \begin{equation} \label{eq/solution_even} \check u(x,\xi) =
    \tilde{u}(x,\xi^2 \pm x^k), \quad \tilde{u}(x,h) = \frac{1}{2}
    \int_{0}^x g_1(s,h \mp s^k) ds,
  \end{equation}
  is a well-defined formal series, which is convergent whenever $g$
  is, and solves the cohomological equation
  \begin{equation} \label{eq/Moser} \{H, \check u\} =
    \xi g_1(x, \xi^2)\,.
  \end{equation}
  (Formula~\eqref{eq/solution_even} is similar to
  \eqref{eq/solution_cohomolog} with the roles of $x$ and $\xi$
  interchanged.)  Now, \eqref{equ:formal} is equivalent to
  $\{H,u-\check u\} = g_0(x,\xi^2) -
  c(x)$. 
  In what follows, we will therefore replace $g(x,\xi)$
  in~\eqref{equ:formal} by $g(x, \xi^2)$.  Thus, consider the
  cohomological equation
  \begin{equation}
    \label{eq/Moser3}
    \{H, u\}= s(x, \xi^2)\,,
  \end{equation}
  where $s = g(x, \xi^2) - c(x)$ with $c = c(x)$ an unknown formal
  series to be determined.
    
  \noindent\textbf{Step 2.} Take
  $u_{ij} = \frac{1}{2 (i+1) }x^{i+1} \xi^{2j-1}$ with
  $i\geq 0, j\geq 1$. Then
  \[
    \{H,u_{i,j}\} = x^{i} \xi^{2j} \mp \frac{(2 j-1) k}{2 (i+1)}
    x^{i+k} \xi^{2j-2}\,.
  \]
  Hence summing up the monomials
  $u_{ij} = \frac{1}{2 (i+1) }x^{i+1} \xi^{2j-1}$ with appropriate
  coefficients, one can eliminate the variable $\xi$ from the
  symplectic form, as follows. For notational convenience we introduce
  $\tau(i,j)=(i+k, j-1)$ and $h(i,j)=\frac{(2j-1)k}{2(i+1)}$, and
  denote $(x\xi)^{(i,j)}:=x^i\xi^{2j}$, so that the previous equation
  writes
  \[
    \{H, u_{i,j}\} = (x\xi)^{(i,j)} \mp h(i,j)(x\xi)^{\tau(i,j)}\,.
  \]
  Notice that the $j$-th iterate $(x\xi)^{\tau^{j}(i,j)} = x^{i+kj}$ depends only
  on $x$. Hence we may solve the equation
  \begin{equation}
    \{H, U_{i,j}\} = (x\xi)^{(i,j)} \mod \RM[x] \qquad  \text{(with } i\geq 0, j\geq 1 \text{)}
    \label{eq:Uij}
  \end{equation}
  using an Ansatz of the form
  \[
    U_{i,j} = \sum_{n=0}^{j-1}a_{ij}^n u_{\tau^{n}(i,j)}\,.
  \]
  In order to determine the coefficients $a_{ij}^n$ we write
  \begin{align}
    \label{equ:HUij}
    \{H, U_{i,j}\}
    & = \sum_{n=0}^{j-1}a_{ij}^n  (x\xi)^{\tau^{n}(i,j)} \mp
      \sum_{n=0}^{j-1}a_{ij}^n h(\tau^{n}(i,j)) (x\xi)^{\tau^{n+1}(i,j)}\\
    & = a_{ij}^0  (x\xi)^{\tau^{0}(i,j)} + \sum_{n=1}^{j-1}\left(a_{ij}^n \mp a_{ij}^{n-1}h(\tau^{n-1}(i,j)) \right)(x\xi)^{\tau^{n}(i,j)}\\
    & \quad \mp a_{ij}^{j-1}h(\tau^{j-1}(i,j)) x^{i+kj}
  \end{align}
  in which the last term on the right belongs to
  $\RM[x]$. Therefore~\eqref{eq:Uij} is satisfied if and only if
  \[
    \begin{cases}
      a_{ij}^0 = 1\\
      a_{ij}^n \mp a_{ij}^{n-1} h(\tau^{n-1}(i,j)) = 0 \quad\text{
        for } 1\leq n\leq j-1\,,
    \end{cases}
  \]
  which gives, for $1\leq n\leq j-1$,
  \[
    a_{ij}^n = (\pm 1)^n \prod_{p=0}^{n-1} h(\tau^{p}(i,j)) = (\pm
    1)^n\frac{(2j-1)k}{2(i+1)}\frac{(2j-3)k}{2(i+1+k)}\cdots \frac{(2j
      -2n+1) k}{2(i+(n-1)k+1)}\,.
  \]
  To summarize, we have obtained
  \begin{align}
    \label{equ:Uij-full}
    \{H, U_{i,j}\} & = (x\xi)^{(i,j)} \mp a_{ij}^{j-1} h(\tau^{j-1}(i,j)) x^{i+kj}\\
   & = (x\xi)^{(i,j)}  - (\pm 1)^j \prod_{p=0}^{j-1} h(\tau^{p}(i,j))x^{i+kj}\,.
  \end{align}
  Hence if we write
  \[
    g (x,\xi^2) = \sum_{0\leq i \atop 0\leq j} g_{i,j}(x\xi)^{(i,j)}
  \]
  we can solve~\eqref{eq/Moser3} formally by
  \[
    u:=\sum _{i,j}g_{i,j}U_{i,j} = \sum_{i,j}\sum_{n=0}^{j-1}
    g_{i,j}a_{ij}^n u_{\tau^{n}(i,j)}\,.
  \]
  For a fixed value $(\tilde \imath, \tilde \jmath)$ of
  $\tau^{n}(i,j) = (i+nk, j-n)$ there are only a finite number of
  relevant indices $n$, namely
  $0\leq n\leq \lfloor \frac{\tilde \imath}{k} \rfloor $. Therefore we
  can write
  \[
    u = \sum_{\tilde \imath, \tilde \jmath}\sum_{n=0}^{\lfloor
      \frac{\tilde \imath}{k} \rfloor} g_{i,j}a_{ij}^n u_{(\tilde
      \imath, \tilde \jmath)}, \quad \mbox{ where } \ \ (i,j) =
    \tau^{-n}(\tilde \imath, \tilde \jmath) = (\tilde\imath - nk,
    \tilde\jmath + n)\,.
  \]
  Let us show that the summation leads to a real-analytic function $u$
  (and hence $c$) solving~\eqref{eq/Moser3}. For this we need to show
  that there exist constants $C,R$ such that
  \[
    \forall \tilde\imath\geq 0, \tilde\jmath\geq 0, \quad
    \abs{\sum_{n=0}^{\lfloor \frac{\tilde \imath}{k} \rfloor}
      g_{i,j}a_{ij}^n } \leq C R^{\tilde\imath + \tilde\jmath}\,.
  \]
  Since we already know that
  $\abs{g_{i,j}}\leq cr^{i+j}\leq c r^{i+kj} = cr^{\tilde\imath +
    k\tilde\jmath}\leq c(r^k)^{\tilde\imath + \tilde\jmath}$, it is
  enough to show that
  \begin{equation}
    \abs{\sum_{n=0}^{\lfloor \frac{\tilde \imath}{k} \rfloor}
      a_{ij}^n } \leq C R^{\tilde\imath + \tilde\jmath}\,.
    \label{equ:sum_aij}
  \end{equation}
  Let us denote $i' = \frac{i+1}{k}$ and $j'=\frac{2j-1}{2}$. We have
  \[
    \abs{a_{ij}^n} =
    \frac{j'}{i'}\frac{(j'-1)}{(i'+1)}\cdots\frac{(j'-n+1)}{(i'+n-1)}
    =
    \frac{\Gamma(j'+1)}{\Gamma(j'-n+1)}\frac{\Gamma(i')}{\Gamma(i'+n)}
  \]
  which can be rewritten in terms of the beta function
  $B(x,y)=\frac{\Gamma(x)\Gamma(y)}{\Gamma(x+y)}$ as
  \[
    \abs{a_{ij}^n} = \frac{B(i',j'+1)}{B(i'+n, j'-n+1)}.
  \]
  It is known that $\abs{B(x,y)}\leq\pi$ for $x,y\geq 1$,
  see~\cite[5.12.2]{NIST2010}, and that
  $\abs{\frac{1}{B(\frac{x+y+1}{2}, \frac{x-y+1}{2})}} \leq 2^{x-1}x$
  for $x\geq 1$, see~\cite[5.12.5]{NIST2010}. This gives, for
  $i', j' \geq 1$
  \begin{equation}
    \abs{a_{ij}^n} \leq 2^{i'+j'-1}(i'+j')\pi \leq CR^{i'+j'}
    \label{equ:radius}  
\end{equation}
  for suitable constants. Since $i'+j'$ is constant in the
  sum~\eqref{equ:sum_aij}, we get that the sum is bounded by
  $(\lfloor \frac{\tilde \imath}{k} \rfloor+1) CR^{i'+j'} \leq \tilde
  C{\tilde R}^{\tilde\imath + \tilde\jmath}$ for suitable
  $\tilde C, \tilde R$, which proves the analyticity of $u$.

  \noindent\textbf{Step 3.} It remains to understand the freedom left to
  simplify the function $c = c(x)$. We may decompose the space
  $\Omega^0$ of analytic functions of $(x,\xi)$ as
  \begin{equation}
    \Omega^0 = \RM\{\xi\} \oplus x \Omega^0\,,
    \label{equ:Omega0}
  \end{equation}
  where $\RM\{\xi\}$ denotes the space of real analytic functions of
  $\xi$. Since $u_{i,j}\in x \Omega^0$, the computation in Step 2
  shows that the map $u\mapsto \{H, u\}$ is surjective from
  $x\Omega^0$ onto $\RM\{x,\xi^2\} \mod \RM\{x\}$. Let us show that
  \begin{equation}
    \{H, \cdot\} : x\Omega^0 \to \RM\{x,\xi^2\} / \RM\{x\}
    \label{equ:isomorphism}
  \end{equation}
  is actually an isomorphism. The equation $\{H, u\}\in\RM\{x\}$
  writes
  \[
    2\xi\partial_x u \mp k x^{k-1}\partial_\xi u = c(x)
  \]
  for some analytic function $c$. Let us write $u=x^mv$,
  $v\in\Omega^0$, where $m\geq 1$ is maximal (assuming $u\neq
  0$). Plugging this Ansatz we see that $c$ must be divisible by
  $x^{m-1}$. Dividing the equation by $x^{m-1}$ we get
  \[
    2\xi(mv + x\partial_x v) \mp k x^k \partial_\xi v = \tilde c(x)\,,
    \quad \tilde c \in\Omega^0\,,
  \]
  which actually implies (taking $x=\xi=0$) that $\tilde
  c(0)=0$. Restricting to $x=0$ we get $2m \xi v(0,\xi)=0$ and hence
  $v(0,\xi)=0$ for all $\xi$, therefore $v=x\tilde v$, which
  contradicts the maximality of $m$. Therefore, $u=0$, which shows
  that $\{H, \cdot\}$ is indeed injective on $x\Omega^0$.

  We may now turn to solutions to the cohomological equation
  $\{H, u\} = c(x)$. We decompose $u = u_0 + u_1$ according
  to~\eqref{equ:Omega0}, so that our equation becomes
  \[
    \{H, u_1\} = c(x) + \{u_0, H\}\,.
  \]
  As in Step 1, we can assume that
  $\{u_0, H\}=\pm kx^{k-1}\partial_\xi u_0$ is even in $\xi$ (notice
  that the function $u$ in~\eqref{eq/solution_even} belongs to
  $x\Omega^0$), which means that
  $u_0(\xi) = \sum_{j\geq 0}a_j \xi^{2j+1}$.  By the
  isomorphism~\eqref{equ:isomorphism}, each choice of $u_0$ leads to a
  unique solution $u_1$, and this solution must hence be given by the
  construction of Step 2. Namely, each monomial $\xi^{2j+1}$ gives by
  $\{H, \cdot\}$ a monomial of type $(x\xi)^{(k-1,j)}$ which is
  absorbed in a unique way by $U_{k-1,j}$ as in~\eqref{eq:Uij},
  leaving on the right-hand side a monomial in $\RM\{x\}$ of the form
  given by the last term of~\eqref{equ:HUij}, \emph{i.e.} a constant
  times $(x\xi)^{\tau^{j}(k-1,j)} = x^{k-1+jk}$. Therefore by
  suitably choosing $u_0$ we may formally eliminate any series of the
  form
  \[
    \sum_{j \ge 0} \tilde a_j x^{k-1+jk} = x^{k-1} \sum_{j \ge 0}
    \tilde a_j x^{jk}
  \]
  and only such series. It is left to observe that if
  $x^{k-1}\tilde c(x^k)$ is real-analytic, then we may choose $u$ to be
  given by the explicit formula (see~\eqref{eq/solution_cohomolog})
  \begin{equation}
    u = \tilde{u}(\xi,H(x,\xi)),
    \quad \mbox{where } \  \tilde{u}(\xi,h) = \frac{\mp 1}{k} \int_{0}^\xi \tilde c(h-s^2) \textup{d}s,
  \end{equation}
  and hence to be real-analytic.
\end{proof}

\begin{proof}[Proof of Theorem~\ref{theo:classical-analytic}]~\\
  By Proposition~\ref{prop:formal}, using Moser isotopies, we can
  transform the symplectic structure to the following form:
 $$
 \psi^*\omega = \sum_{i=0}^{k-2} x^i c_i(x^k) \textup{d} \xi \wedge
 \textup{d} x = \dd (f \textup{d} \xi), \ f = \sum_{i=1}^{k-1} x^i
 f_i(x^k).
  $$
  According to Lemma~\ref{lemm:A} and
  Proposition~\ref{prop/character_moser_type}, further simplification
  of the normal form is possible only by using linear transformations
  $A$; this is effective only when $k$ is even, in which case the even
  part of $f$ is defined up to sign. Indeed, in this case
  $\textup{Inv}(x,\xi) = (-x,-\xi)$ is an $H$-preserving
  diffeomorphism that is not of Moser type and
  $\textup{Inv}^*\psi^*\omega = \textup{d}(-f(-x) \textup{d}\xi)$. 
  
  The conclusion is that we have found
  exactly one representative for every orbit of the group of 
  real-analytic diffeomorphism germs.
  The result follows.
\end{proof}

In a similar way, one can prove the following result, which is useful
for the fiber-wise symplectic classification problem. See also \cite[Theorem 4]{Francoise1978}
(or \cite[Theorem 2.7]{Francoise1988})
for the case of general isolated singularities of (complex-)analytic functions; our result
can be viewed as a specification of this general theorem to $A_k$ singularities.

\begin{theo}
  \label{theo:classical-analytic_H}
  Let $H$ be a real-analytic Hamiltonian on $(\RM^2, \omega)$
having $A_{k-1}, k \ge 2,$ singularity at the origin $O$.
  Then the pair $(H, \omega)$ has also the following local symplectic normal form near $O$ for the group of
  $H$-preserving real-analytic diffeomorphism germs:
  \begin{equation}
    \big(H = \xi^2 \pm x^{k}, \ \omega = \sum_{i=0}^{k-2} x^i \tilde{c}_i(H) \textup{d} \xi
  \wedge \textup{d} x \big)
  \end{equation}
  for some real-analytic functions $\tilde{c}_i$ with $\tilde c_0(0) \ne
  0$ defined uniquely up to simultaneously changing the sign of the functions
    $\tilde c_{2i+1}(H), \ i = 0, \ldots, \ell -2,$ when $k = 2\ell$ is even. In other words,
      \begin{itemize}
  \item if $k$ is odd, then (the Taylor series of) $\tilde c_{i}, i = 0, \ldots, k -2,$ form a complete
  and minimal set of symplectic invariants;
  \item if $k = 2\ell$ is even, then (the Taylor series of) $\tilde c_{i}, i = 0, \ldots, k -2,$ form a complete
    set of symplectic invariants with the only relation given by changing the sign of 
    $\sum_{i=0}^{\ell-2} x^{2i+1} \tilde c_{2i+1}(H)$.
  \end{itemize}
\end{theo}
\begin{proof}
  By Theorem~\ref{theo:classical-analytic}, we can assume that
  $g = \sum_{i=0}^{k-2} x^i c_i(x^k)$.  We would like to solve the
  cohomological equation
  \begin{equation} \label{eq/cohomological_H} \dd H \wedge \dd u =
    (g(x) - \tilde c(x, H))\textup{d} \xi \wedge \textup{d} x, \qquad
    \tilde c = \sum_{i=0}^{k-2} x^i \tilde{c}_i(H),
  \end{equation}
  where $\tilde c_i = \tilde c_i(H)$ are real-analytic functions to be
  determined later.  In terms of the Poisson bracket, equation
  \eqref{eq/cohomological_H} reads
  \begin{equation} \label{eq/cohomological_Poisson_H} \{H, u\} = g(x)
    - \tilde c(x, H) = \sum_{i=0}^{k-2} \sum_{j \ge 0} x^i (c_{ij}
    x^{jk} - \tilde{c}_{ij} H^{j}),
  \end{equation}
  where $c_i = \sum_{j \ge 0} c_{ij} x^{jk} $ and
  $\tilde{c}_i = \sum_{j \ge 0} \tilde{c}_{ij} H^{j}$.

  Recall that for $u_{ij} = \frac{1}{2 (i+1) }x^{i+1} \xi^{2j-1}$ with
  $i\geq 0, j\geq 1$, we have that
  \begin{equation} \label{eq/monomial_H} \{H, u_{i,j}\} = x^{i}
    \xi^{2j} \mp \frac{(2 j-1) k}{2 (i+1)} x^{i+k} \xi^{2j-2}\,.
  \end{equation}
  It follows that for every $i = 0, \ldots, k-2,$ and $j \ge 0,$ there
  exists a polynomial $U_{ij}$ and a unique constant
  $b_j, \ |b_j| \ge 1,$ such that
  \[
  \{H, U_{ij}\} =   b_j x^{i+jk} -  x^i H^j\,.
  \]
  Indeed, one has
  \[
    x^i H^j = \sum_{n=0}^j(\pm 1)^n \binom{j}{n} (x\xi)^{\tau^n(i,j)}
\]
and using~\eqref{equ:Uij-full} we see that each term in this sum will
contribute to a coefficient of $x^{i+kj}$ of sign
$(\pm 1)^n \times (\pm 1)^{j-n} = (\pm 1)^j$, independent of $n$.

  Hence setting
  \[
    U := \sum_{i=0}^{k-2} \sum_{j \ge 0} \frac{c_{ij}}{b_j}U_{ij},
  \]
  we get that
  \[
    \{H, U\} = g(x) - \sum_{i=0}^{k-2} \sum_{j \ge 0}
    \frac{c_{ij}}{b_j} x^i H^j.
  \]
  This shows that equations
  (\ref{eq/cohomological_H}-\ref{eq/cohomological_Poisson_H}) can be
  solved formally. To show that
  (\ref{eq/cohomological_H}-\ref{eq/cohomological_Poisson_H}) admits a
  real-analytic solution, one can estimate the coefficients of $U$
  similarly to how this is done in the proof of
  Theorem~\ref{theo:classical-analytic}. However, a simpler proof can
  be obtained as follows. First note that by Eq.\eqref{eq/monomial_H},
  the constants $b_j$ are always such that $|b_j| \ge 1$.  Hence the
  series
$$
g = \sum_{i=0}^{k-2} \sum_{j \ge 0} c_{ij} x^{i+jk}
$$
being convergent implies that so is the series
$$
\tilde c(x, H) = \sum_{i=0}^{k-2} \sum_{j \ge 0} \tilde{c}_{ij} x^{i}
H^{j},
$$
where $\tilde{c}_{ij} = c_{ij}/b_j$. But Proposition~\ref{prop:formal}
now implies that we can solve the cohomological equation
$$
 \{H, \tilde U\} = c(x) - \tilde c(x, H)\,,
$$
where $c = c(x)$ is the unique normal form of
Proposition~\ref{prop:formal}, which is convergent.  But, by the
construction of $\tilde c$, the function $g$ solves the latter
cohomological equation formally. Hence, by uniqueness, $c=g$.

Thus, $u= \tilde U(x, \xi)$ is an analytic function solving
\eqref{eq/cohomological_H}.  The result follows.
\end{proof}

As a corollary, we get

\begin{theo}
  \label{theo:classical-analytic_fibration_preserving_case}
   Let $H$ be a real-analytic Hamiltonian on $(\RM^2, \omega)$
having $A_{k-1}, k \ge 2,$ singularity at the origin $O$.
  Then under the fibration-preserving equivalence,   
  the pair $(H, \omega)$ has the following local symplectic analytic normal form near $O$:
  \begin{equation}
    \big(H = \xi^2 \pm x^{k}, \ \omega = (1+\sum_{i=1}^{k-2} x^i \hat{c}_i(H)) \textup{d} \xi
  \wedge \textup{d} x \big)
  \end{equation}
  for some real-analytic functions $\hat{c}_i$  defined uniquely up to simultaneously changing the sign of the functions
    $\hat c_{2i+1}(H), \ i = 0, \ldots, \ell -2,$ when $k = 2\ell$ is even. In other words,
      \begin{itemize}
  \item if $k$ is odd, then (the Taylor series of) $\hat c_{i}, i = 1, \ldots, k -2,$ form a complete
  and minimal set of symplectic invariants;
  \item if $k = 2\ell$ is even, then (the Taylor series of) $\hat c_{i}, i = 1, \ldots, k -2,$ form a complete
    set of symplectic invariants with the only relation given by changing the sign of 
    $\sum_{i=0}^{\ell-2} x^{2i+1} \hat c_{2i+1}(H)$.
  \end{itemize}
  \end{theo}
\begin{proof}
Let
$\tilde H$ be a real-analytic Hamiltonian  defining the same (singular) foliation as $H$ near the origin. Then
$\tilde H = h(H)$ for an analytic germ $h$ with $h'(0) \ne 0.$ Equivalently, we can write $h = H g(H)$ for
an analytic germ $g$ with $g(0) \ne 0.$ The fibration-preserving classification can be obtained by 
simplifying  the  $H$-preserving normal of $(H, \omega)$  by declaring
$(\tilde H = H g(H), \omega)$ and $(H, \omega)$ to be equivalent. On the level of $H$-preserving normal forms,
this can specifically be done as follows.

Assume (without loss of generality) that $g(0) > 0$ and change the variables according to
$$\psi(x,\xi) = (\tilde x, \tilde \xi), \quad \tilde \xi = \xi \sqrt{g(H)}, \ \tilde x = x (g(H))^{1/k}.$$ Then
$Hg(H) = {\tilde \xi}^2 \pm {\tilde x}^k$ and it follows that the pairs $(H, \omega)$ and $(H, \tilde \omega = (\psi^{-1})^{*} \omega)$
are equivalent under the fibration-preserving equivalence. In other words, the normal forms
 $$
 (h(H) = {\tilde \xi}^2 \pm {\tilde x}^k, \ \tilde \omega = \sum_{i=0}^{k-2} \tilde x^i \hat{c}_i(h(H)) \textup{d} \tilde \xi
 \wedge \textup{d} \tilde x) \quad \mbox{ and } \quad (H = {\xi}^2 \pm {x}^k, \ \omega =
 \sum_{i=0}^{k-2} x^i c_i(H) \textup{d} \xi \wedge \textup{d} x),
  $$  
  where
  \begin{equation}
    c_i(H) = \hat{c}_i(h(H)) g(H)^{\frac{i+1}{k} - \frac{1}{2}} (g(H)+g'(H) H) = \hat{c}_i(h(H)) g(H)^{\frac{i+1}{k} - \frac{1}{2}} h'(H),\label{equ:hat_c_i}  
  \end{equation}
  are equivalent under the fibration-preserving equivalence.
  
  Note that by applying the orientation-reversing diffeomorphisms
  $(x,\xi) \mapsto (x, -\xi)$ if necessary, we can achieve that
  $\hat{c}_0(0) > 0$.  Then we can uniquely solve the equation
  \begin{equation} \label{eq/ode} h'(H) \, |h(H)|^{\frac{1}{k} -
      \frac{1}{2}} \hat{c}_0(h(H)) = |H|^{\frac{1}{k} - \frac{1}{2}}
  \end{equation}
  for $h$ with the condition $h(0) = 0$, making $c_0(H) \equiv
  1$. Indeed, Eq.~\eqref{eq/ode} with the condition $h(0) = 0$ is
  equivalent to the integral equation
  $$
  \int_0^{h(H)} |h|^{\frac{1}{k} - \frac{1}{2}} \hat{c}_0(h)
  \textup{d} h = \int_0^{H} |H|^{\frac{1}{k} - \frac{1}{2}} \textup{d}
  H = \textup{sign}(H) \frac{2k}{2+k} |H|^{\frac{1}{k} + \frac{1}{2}},
  $$
  which has the form
  $\textup{sign}(h) \, |h|^{\frac{1}{k} + \frac{1}{2}} f(h) =
  \textup{sign}(H) |H|^{\frac{1}{k} + \frac{1}{2}}$ for some smooth
  function $f = f(h)$ with $f(0) = \hat{c}_0(0) > 0.$ Equivalently, we
  can write this equation as
  $$
  h \hat f(h) - H = 0, \quad \hat f = f(h)^{\frac{2k}{2+k}}, \ \hat
  f(0) > 0,
  $$
  which admits a unique real-analytic solution $h = h(H)$ near the
  origin by the Implicit Function Theorem. The result follows.
\end{proof}
\begin{remark}
  Note that orientation-reversing diffeomorphisms (e.g.
  $\psi(x,\xi) = (x, -\xi)$) in
  Theorem~\ref{theo:classical-analytic_fibration_preserving_case} are
  allowed. If we restrict to orientation-preserving diffeomorphisms,
  then the normal form is given by
  $$
\big(H = {\xi}^2 \pm {x}^k, \ \pm 1 + \sum_{i=1}^{k-2} x^i \hat{c}_i(H)\big)
  \textup{d} \xi \wedge \textup{d} x),
  $$
  where the sign is plus if
  $\textup{d} \xi \wedge \textup{d} x/ \omega > 0$ and minus if
  $\textup{d} \xi \wedge \textup{d} x/\omega < 0$.
\end{remark}

\section{Symplectic classification in the $C^\infty$-smooth
  case} \label{sec/classification_smooth}

We shall see in this section that in the $C^\infty$-smooth case, the
symplectic classification works qualitatively differently depending on
whether $H$ is a proper map in a neighbourhood of the origin. It turns
out that for the `compact' $A_{2\ell -1}$ singularities for which $H$
can be written in the form
$$H = \xi^2 + x^{2\ell}, \ \ell > 1,$$ the symplectic invariants are given by a collection of both function germs
and Taylor series\footnote{In the classical elliptic case $\ell = 1$,
  we have only one function germ, the action variable, as the
  symplectic invariant for the $H$-preserving equivalence and no
  symplectic invariants in the fibration-preserving case; see
  \cite{DeVerdiere1979}}, whereas for `non-compact' $A_k$
singularities --- which can then be expressed as the isolated
singularity of the Hamiltonian $H = \xi^2 - x^k,$ with $k$ arbitrary
--- the symplectic invariants may be expressed in terms of Taylor
series only; see Theorems~\ref{theo:classical-smooth}, \ref{theo:classical-smooth_H} and
\ref{theo:classical-smooth_fibration_preserving_case} below.

In the latter case of $H = \xi^2 - x^k,$ we will need the following
result.

\begin{theo} \label{theorem/remaining_cases} Let $H = \xi^2 - x^k$
  with $k \ge 2$ and $\omega_0$ and $\omega_1$ be arbitrary symplectic
  forms near $0\in\RM^2$.  Assume that the corresponding local action
  variables are equal up to a $C^\infty$-smooth function $f$:
$$
I_0(h) := \int\limits_{R(h)} \omega_0 = I_1(h) - f(h) :=
\int\limits_{R(h)} \omega_1 - f(h),
$$
where the region $R(h)$ is enclosed by the level curves $H = 0,$
$x = (\xi^2 - h)^{1/k}, \ h \le 0,$ and $\xi = \pm \varepsilon;$ see
Figure~\ref{Action_vars}.  Then there exists a local $C^\infty$-smooth
$H$-preserving isotopy $\psi_t$ of Moser type, such that
$
  \psi_1^* \omega_1 = \omega_0\,.
$
\end{theo}
\begin{figure}[htbp]
  \begin{center}
    \includegraphics[width=0.7\linewidth]{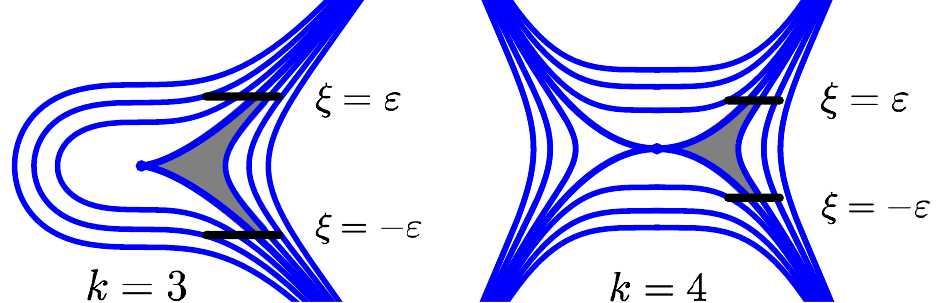}
    \caption{$H = \xi^2 - x^{k}$ singularities for $k = 3$ (left) and
      $k = 4$ (right).  The region $R(h)$ is shaded gray.}
    \label{Action_vars}
  \end{center}
\end{figure}
\begin{proof}
  The idea of the proof is to use the smoothness of (the derivative
  of) the difference of the local action variables
  \[
    f'(h) = \int_{-\varepsilon}^{\varepsilon} \frac{g((\xi^2 -
      h)^{\frac{1}{k}},\xi)}{k (\xi^2 - h)^{\frac{k-1}{k}}} \dd \xi
  \]
  to show the smoothness of the solution
  \begin{equation}
    u(x,\xi) = \tilde{u} (\xi, \xi^2 - x^k), \quad \mbox{ where } \quad \tilde u(\xi,h) = \int^\xi_{-\varepsilon} \dfrac{g((t^2-h)^{1/k},t)}{k(t^2-h)^{\frac{k-1}{k}}} dt,
  \end{equation}
  to the cohomological equation
  \begin{equation}
    \dd H \wedge \dd u = \omega_1 - \omega_0 =: g(x,\xi) \dd \xi \wedge \dd x.
  \end{equation}
  This can be done similarly to the case of the cusp ($k = 3$)
  singularity; see \cite{Kudryavtseva2021}.  We give a complete proof
  in the Appendix.
\end{proof}

\begin{remark}
  We note that an analogous result is false for `compact'
  $A_{2\ell-1}$ singularities $H = \xi^2 + x^{2\ell}$ when $\ell \ge 2$: (the
  asymptotics of the) action variables do not suffice for their
  symplectic classification, neither in the $H$-preserving nor in the
  fibration preserving case.  This follows from
  Theorems~\ref{theo:classical-smooth} and
  \ref{theo:classical-smooth_fibration_preserving_case} below.  As a
  concrete example, one can take
$$
\big(H = \xi^2 + x^{2\ell},\, \omega_0 = \dd \xi \wedge \dd x\big)
\quad \mbox{ and } \quad \big(H = \xi^2 + x^{2\ell},\, \omega_1 =
(1+xc_1(x^{2\ell}))\dd \xi \wedge \dd x\big),
$$
where $\ell \ge 2$ and $c_1$ is an arbitrary function whose Taylor
series at the origin does not vanish identically. Equivalently, the
following two integrable Hamiltonian systems
$$
(H_0 = \xi^2 + x^{2\ell},\, \dd \xi \wedge \dd x) \quad \mbox{ and }
\quad (H_1 = \xi^2 + V(\tilde x),\, \dd \xi \wedge \dd {\tilde x}),
$$
where the potential $V(\tilde x) = x(\tilde x)^{2\ell}$ for $\tilde x = x+ f_1(x)$ with $f_1(x) = \int_0^x sc_1(s^{2\ell}) \dd s$ and $ \ell \ge 2$, are not
($C^\infty$-smoothly or real-analytically) fiberwise symplectically
equivalent even though their (\textit{real}) action variables are
identically the same.
\end{remark}

To generalise the normal form results of the previous subsection to
the $C^\infty$-smooth case, we will also need to find a
$C^\infty$-smooth solution to a certain integral equation as in the
following lemma:

\begin{lemm} \label{lemma/integral_Abel} Let
  $G \colon [0,h_0] \to \mathbb R$ be a $C^\infty$-smooth function on
  $[0,h_0], \, h_0>0,$ and $k \ge 2$ be a natural number. Then for
  each $i=0,1, \ldots, k-2$, the equation
  \[
    \int_{-1}^{1} \dfrac{c_i(h(1-\xi^2))}{(1-\xi^2)^{\frac{k-1-i}{k}}}
    \dd \xi = G(h)
  \]
  admit a $C^\infty$-smooth solution $c_i\in\Cinf([0,h_0])$ (with an
  ``explicit'' formula).
\end{lemm}
\begin{proof}
  Letting $u=1-\xi^2$ for $\xi\in[0,1]$ we obtain
  \[
    \int_{-1}^{1} \frac{c_i(h(1-\xi^2))}{(1-\xi^2)^{\frac{k-1-i}{k}}}
    \dd \xi = \int_0^1 \frac{c_i(hu)}{u^{\frac{k-1-i}{k}}}\frac{\dd
      u}{\sqrt{1-u}}.
  \]
  Introducing $\tilde c_i(t) := t^{\frac{i+1-k}{k}}c_i(t)$ and
  $\tilde G(h):= h^{\frac{i+1-k/2}{k}}G(h)$, our equation becomes
  \[
    \int_0^h \frac{\tilde c_i(t)}{\sqrt{h-t}}\dd t = \tilde G(h)\,.
  \]
  We recognize the Abel transform. Since $\tilde G \in L^1(0;h_0)$ we
  may write the inverse Abel transform \cite{Abel1826} to obtain:
  \[
    \tilde c_i(t) = \frac{1}{\pi} \frac{\dd{}}{\dd t} \int_0^t
    \frac{\tilde G(h)}{\sqrt{t-h}}\dd h\,.
  \]
  Note that the derivative of the integrand is not $L^1$, but we can
  write
  \begin{align}
    \int_0^t \frac{\tilde
    G(h)}{\sqrt{t-h}}\dd h
    & = t^{1/2}\int_0^1 \frac{\tilde G(tv)}{\sqrt{1-v}}\dd v \\
    & = t^{\frac{i+1}{k}}\int_0^1 v^{\frac{i+1-k/2}{k}} \frac{G(tv)}{\sqrt{1-v}}  \dd v\,,
  \end{align}
  in which form it is now licit to take the derivative under the
  integral sign. We have
  \begin{multline}
    \frac{\dd{}}{\dd t}
    \left(t^{\frac{i+1}{k}}\int_0^1 v^{\frac{i+1-k/2}{k}} \frac{G(tv)}{\sqrt{1-v}}\dd v \right) \\
    = t^{\frac{i+1-k}{k}}\left(\frac{i+1}{k}\int_0^1
      v^{\frac{i+1-k/2}{k}} \frac{G(tv)}{\sqrt{1-v}} \dd v + t\int_0^1
      v^{\frac{i+1+k/2}{k}} \frac{G'(tv)}{\sqrt{1-v}} \dd v \right).
  \end{multline}
  Hence, we obtain
  \[
    c_i(t) = t^{\frac{k-i-1}{k}} \tilde c_i(t) =
    \frac{1}{\pi}\left(\frac{i+1}{k}\int_0^1 v^{\frac{i+1-k/2}{k}}
      \frac{G(tv)}{\sqrt{1-v}} \dd v + t\int_0^1 v^{\frac{i+1+k/2}{k}}
      \frac{G'(tv)}{\sqrt{1-v}}\dd v \right),
  \]
  and we see that the right-hand side is smooth.
\end{proof}

We now turn to the right symplectic classification in the smooth
category.

\begin{theo}
  \label{theo:classical-smooth}
  Let $H$ be a $C^\infty$-smooth Hamiltonian on $(\RM^2, \omega)$
having $A_{k-1}, k \ge 2,$ singularity at the origin $O$.
  Then the pair $(H, \omega)$ has the following local symplectic normal form near $O$ for the group of
  $H$-preserving $C^\infty$-smooth diffeomorphism germs:
  \begin{equation}
    \big(H = \xi^2 \pm x^{k}, \ \omega = \textup{d}(f \textup{d} \xi)\big), \quad \mbox{ where } f =
    \sum_{i=1}^{k-1} x^i f_i(x^k),\label{equ:normal-form}
  \end{equation}
  for some $C^\infty$-smooth functions $f_i$ with $f_1(0) \ne
  0$ defined uniquely modulo the following relations:
  \begin{itemize}
  \item if $k = 2\ell + 1$ is odd, then $f_{i}$'s are uniquely defined up to addition of flat functions (that is,
  the Taylor series of $f_{i}$'s form a complete and minimal set of symplectic invariants);
  \item if $k = 2\ell$ is even and $H = \xi^2 - x^{2\ell}$, then $f_{i}$'s are uniquely defined up to addition of flat functions 
  and changing the
    sign of $\sum_{i=1}^{\ell-1} x^{2i} f_{2i}(x^{2\ell})$;
  \item if $k = 2\ell$ is even and $H = \xi^2 + x^{2\ell}$, the even-indexed functions $f_{2i}, 1 \le i \le \ell-1,$
    are uniquely defined up to addition of flat functions and changing the sign of
    $\sum_{i=1}^{\ell-1} x^{2i} f_{2i}(x^{2\ell})$.
  \end{itemize}
\end{theo}
\begin{remark}
  Thus, the only difference with the real-analytic case appears for
  $A_{2\ell-1}$ singularities given by
  $H = \xi^2 + x^{2\ell}, \ \ell \ge 1$, in which case these are the
  {\em germs} of the functions $f_{2i+1}, 0 \le i \le \ell-1,$ rather than
  their Taylor series that are the symplectic invariants.
\end{remark}
\begin{proof} \mbox{ }

 \noindent\textbf{Step 1.1.}  The first two cases ($k$ odd and $k = 2\ell$ even with $H = \xi^2 - x^{2\ell}$, respectively) follow
 from the analytic case using Borel summation and
 Theorem~\ref{theorem/remaining_cases} above. Indeed, one can take
 $f = f(x)$ to be any smooth function such that its Taylor series at
 the origin gives the formal normal form for $\omega.$ Then
 $\omega_0 = \textup{d}(f \textup{d} \xi)$ and $\omega_1 = \omega$
 satisfy the conditions of Theorem~\ref{theorem/remaining_cases}.

 More specifically, we first observe that there exists a smooth
 function $u$ that solves the cohomological equation
$$
\dd H \wedge \dd u = \omega_1 - \omega_0 \mod \quad \mathcal O(\infty)
\dd \xi \wedge \dd x,
$$
i.e., up to a flat 2-form. This follows from the formal solution to
this equation constructed in Theorem~\ref{theo:classical-analytic} by
applying the Borel summation. Thus without loss of generality,
$$
\omega_1 - \omega_0 = g(x,\xi) \dd \xi \wedge \dd x,
$$
where $g = g(x,\xi)$ is a $C^\infty$-smooth flat function germ at the
origin. But then the corresponding `local' action variables (see
Theorem~\ref{theorem/remaining_cases}) satisfy
$$
I_1(h) = I_0(h) + f(h), \quad f(h) :=
\int^{\varepsilon}_{-\varepsilon}
\dfrac{g((t^2-h)^{1/k},t)}{k(t^2-h)^{\frac{k-1}{k}}} dt.
$$
where $f = f(h)$ is $C^\infty$-smooth on $[-h_0, 0]$ with $h_0$ is
sufficiently small: Indeed, since $g(x,\xi)$ is flat at the origin, we
can write $g = x^N \tilde g(x,\xi)$ with $N$ arbitrary large.  Then we
see that
\[
  f(h) = \frac{1}{k} \int^{\varepsilon}_{-\varepsilon}
  (t^2-h)^{\frac{N+1-k}{k}} g((t^2-h)^{1/k},t) dt
\]
is at least $C^{[\frac{N}{k}] - 1}$-differentiable. Letting
$N \to \infty$, we get that $f = f(h)$ is $C^\infty$-smooth. Hence we
can apply Theorem~\ref{theorem/remaining_cases}.

 \noindent\textbf{Step 1.2.}  To conclude the proof in the case when $k$ is odd or $k = 2\ell$ is even with $H = \xi^2 - x^{2\ell}$, it is therefore left to show the 
 uniqueness of the Taylor series of $f_i$ up to changing the sign of
 $\sum_{i=1}^{\ell-1} x^{2i} f_{2i}(x^{2\ell})$ when $k = 2\ell$ is
 even.  This follows from the proof in the analytic case by virtue of
 Lemma~\ref{lemm:A} and Proposition~\ref{prop/character_moser_type}.

 \noindent\textbf{Step 2.1.} In the remainder of the proof we consider the case of a `compact' $A_{2\ell -1}$ singularity, i.e., the case when $k = 2l$ and $H =  \xi^2 + x^{2l}.$ 
 For simplicity, we shall assume that $H = \xi^2 + x^{4}$ (the proof
 for larger $k = 2 \ell$ is similar).

 Observe that without loss of generality, $g = g(x, \xi^2)$ is an even
 function of $\xi$. Indeed, formula \eqref{eq/solution_even} works
 equally well in the smooth case.  Let
 $$g(x, \xi^2) = g_0(x^4, \xi^2) + xg_1(x^4, \xi^2) + x^2 g_2(x^4,
 \xi^2) + x^3 g_3(x^4, \xi^2).$$ We can similarly assume that
 $g_3 \equiv 0$: letting
 \begin{equation}
   v = \tilde{v_4}(\xi,H(x,\xi)), \quad \mbox{where } \  \tilde{v}(\xi,h) = \frac{-1}{4} \int_{0}^\xi \tilde g_3(h-s^2, s^2) \textup{d}s,
 \end{equation}
 we get that $\{H, v\} = x^3 g_3(x^4, \xi^2).$ Hence we can subtract
 this expression from $g$ using Lemma~\ref{lemm:moser}.

 Now observe that one can find functions
 $c_i = c_i(x^4), 0 \le i \le 2,$ depending on $x$ only and such that
 \begin{equation} \label{eq/generalised_actions}
   \int_{-h^{1/2}}^{h^{1/2}} \frac{c_i(h-\xi^2)}{(h-\xi^2)^{(3-i)/4}}
   \textup{d} \xi = \int_{-h^{1/2}}^{h^{1/2}} \frac{g_i(h-\xi^2,
     \xi^2)}{(h-\xi^2)^{(3-i)/4}} \textup{d} \xi,
 \end{equation}
 which is equivalent to
 $$ \int_{-1}^{1} \frac{c_i(h(1-\xi^2))}{(1-\xi^2)^{(3-i)/4}} \textup{d} \xi = \int_{-1}^{1} \frac{g_i(h(1-\xi^2), h\xi^2)}{(1-\xi^2)^{(3-i)/4}} \textup{d} \xi. $$
 The right-hand-side of the last equation is a smooth function of $h$
 since $g$ is smooth; solving the integral equation produces a smooth
 solution $c_i$ by Lemma~\ref{lemma/integral_Abel}.

 We would now like to show that there exists a smooth solution $u_i$
 to the cohomological equation
 \begin{equation} \label{eq/Moser2} \textup{d}H \wedge \textup{d}u =
   x^i(g_i(x^4, \xi^2) - c_i(x^4)) \textup{d} \xi \wedge \textup{d} x.
 \end{equation}
 Indeed, such an $u_i$ is given by
 \begin{equation} \label{eq/solution2} u_i(x,\xi) =
   \tilde{u_i}(\xi^2+x^k, \xi), \quad \tilde{u_i}(h, \xi) =
   -\frac{1}{4} \int_{-\sqrt{h}}^\xi \dfrac{g_i(h - s^2,s^2) -
     c_i(h-s^2)}{(h-s^2)^{(3-i)/4}} ds,
 \end{equation}
 which can be shown to be smooth by the construction of $c_i$;
 cf. \cite{Kudryavtseva2021}. Indeed, let
 $\alpha(x,\xi^2) = x^i(g_i(x^4,\xi^2) - c_i(x^4))$. Then for an
 arbitrary large $n \in \mathbb N$, there exist $C^\infty$-smooth
 functions $R$ and $v$ such that
 \begin{equation}
   \alpha(x,\xi^2) = \sum_{k = 0}^n B_{k} x^k + x^{n+1} R(x,\xi^2) + \textup{d}H \wedge \textup{d}v.
 \end{equation}
 The proof of this is similar to the one used in the analytic case
 (see Step 2 of the proof of Theorem~\ref{theorem/remaining_cases} for
 details).  Recalling that
 $\alpha(x,\xi^2) = x^i(g_i(x^4,\xi^2) - c_i(x^4))$, we see that in
 fact
 \begin{equation}
   \alpha(x,\xi^2) = \sum_{k = 0}^n B_{k} x^{i+4k} + x^{n+1} R_2(x^4,\xi^2) + \textup{d}H \wedge \textup{d} v_2.
 \end{equation}
 Thus, without loss of generality, we can assume that
 \begin{equation}
   \alpha(x,\xi^2) = \sum_{k = 0}^n B_{k} x^{i+4k} + x^{n+1} R_2(x^4,\xi^2)
 \end{equation}
 since $\textup{d}H \wedge \textup{d}u_2$ does not affect the equality
 of the generalised actions \eqref{eq/generalised_actions} by
 integration by parts:
 $$
 \int_{-h^{1/2}}^{h^{1/2}} \frac{\partial_x u \cdot \partial_\xi H -
   \partial_\xi u \cdot \partial_x H}{-4x^3} \textup{d} \xi =
 \int_{-h^{1/2}}^{h^{1/2}} \textup{d}u((h- \xi^2)^{1/4}, \xi) = 0.
 $$
 We have that $B_k = 0$ must be equal to zero, since
 \begin{equation}
   \int_{-h^{1/2}}^{h^{1/2}} \frac{\alpha(x,s^2)}{(h-\xi^2)^{3/4}} \textup{d} \xi = 0
 \end{equation}
 by construction. Letting $n \to \infty$, we conclude that $u_i$ is
 $C^\infty$-smooth.

 \noindent\textbf{Step 2.2.}  Now observe that $c_0$ and $c_2$ are
 uniquely defined since they correspond to the action variable
 $$I'(h) = 2 \int_{-h^{1/2}}^{h^{1/2}} \frac{g(h-\xi^2, \xi^2)}{4(h-\xi^2)^{3/4}} \textup{d} \xi.$$
 That the Taylor series of $c_1$ is uniquely defined up to sign
 follows from the proof in the analytic case and Lemma~\ref{lemm:A}
 and Proposition~\ref{prop/character_moser_type}.  Indeed, the
 integral equation
 $$ \int_{-h^{1/2}}^{h^{1/2}} \frac{c_i(h-\xi^2)}{(h-\xi^2)^{(k-i-1)/k}} \textup{d} \xi = \int_{-h^{1/2}}^{h^{1/2}} \frac{g_i(h-\xi^2, \xi^2)}{(h-\xi^2)^{(k-i-1)/k}} \textup{d} \xi$$
 leads to the same normal form $c = \sum_{i = 0}^{2}x^{i} c_{i}(x^4)$
 as does the elimination procedure in the analytic case. This follows
 from the fact that
 the transformations
 $\omega \mapsto \omega - \textup{d} a_{ij} u_{ij} \wedge \textup{d}
 H,$ where $a_{ij} \in \mathbb R$ are coefficients and $u_{ij}$ are
 the monomials $u_{ij} = \frac{1}{2 (i+1) }x^{i+1} \xi^{2j-1},$ do not
 change the `generalized' actions
  $$\int_{-h^{1/2}}^{h^{1/2}} \frac{g_i(h-\xi^2, \xi^2)}{(h-\xi^2)^{(k-i-1)/k}} \textup{d} \xi$$
  (again by integration by parts). It is only left to observe that we
  can change sign of $c_1(x^4)$ by applying the diffeomorphism
  $\textup{Inv}(x,\xi) = (-x,-\xi)$ and that changing $c_1 = c_1(x^4)$
  by a flat function does not affect the smoothness of the solution
  $u_1$. This shows that the only symplectic invariants are indeed the
  Taylor series of $c_1$ at the origin taken up to sign and the germs
  of $c_0$ and $c_2$. Thus, we can transform the
  symplectic structure to the form
 $$
 \psi^*\omega = \dd (f \textup{d} \xi), \ f = \sum_{i=1}^{3} x^i
 f_i(x^4),
  $$
  where $f'(x) = \sum_{i=0}^{2} x^i c_i(x^4)$ and $f_i$ satisfy the
  uniqueness conditions given in the theorem.
  To conclude the proof, it is left to observe that $f_i$ with these uniqueness conditions
  are indeed symplectic invariants, that is, two pairs 
  $(H, \omega)$ and $(H, \tilde \omega)$
  are related by an $H$-preserving diffeomorphism if and only if the corresponding invariants coincide. We have already proven the
  only if statement. The if statement is shown similarly to Step 2.1: The cohomological equations now
  read
   \begin{equation} \textup{d}H \wedge \textup{d}u =
   x^i(c_i(x^4) - \tilde c_i(x^4)) \textup{d} \xi \wedge \textup{d} x, \quad i = 0,1,2,
 \end{equation} 
 where $c_i(x^4) - \tilde c_i(x^4) = 0$ for $i = 0,2$ and $c_1(x^4) - \tilde c_1(x^4)$ is flat at the origin.
 For $i = 0,2,$ the equations are thus solved trivially, whereas for $i = 1,$ the formula
  \begin{equation}  u_1(x,\xi) =
   \tilde{u_1}(\xi^2+x^k, \xi), \quad \tilde{u_1}(h, \xi) =
   -\frac{1}{4} \int_{-\sqrt{h}}^\xi \dfrac{c_1(h-s^2) - \tilde c_1(h-s^2)}{(h-s^2)^{1/2}} ds,
 \end{equation}
 gives as $C^\infty$-smooth solution since $c_1(x^4) - \tilde c_1(x^4)$ is flat at the origin.
\end{proof}

One can show using the same argument that an $A_{k-1}, k \ge 2,$
singularity has also an $H$-preserving normal form as in the following

\begin{theo}
  \label{theo:classical-smooth_H}
  Let $H$ be a $C^\infty$-smooth Hamiltonian on $(\RM^2, \omega)$
having $A_{k-1}, k \ge 2,$ singularity at the origin $O$.
 Then the pair $(H, \omega)$ has also the following local symplectic normal form near $O$ for the group of
  $H$-preserving $C^\infty$-smooth diffeomorphism germs:
  
  \begin{equation}
    \big(H = \xi^2 \pm x^{k}, \ \omega = \sum_{i=0}^{k-2} x^i c_i(H) \textup{d} \xi
  \wedge \textup{d} x \big)
  \end{equation}
  for some $C^\infty$-smooth  functions $c_i$  defined  uniquely modulo the following relations:
  \begin{itemize}
  \item if $k = 2\ell + 1$ is odd, then $c_{i}$'s are uniquely defined up to addition of flat functions (that is,
  the Taylor series of $c_{i}$'s form a complete and minimal set of symplectic invariants);
  \item if $k = 2\ell$ is even and $H = \xi^2 - x^{2\ell}$, then $c_{i}$'s are uniquely defined up to addition of flat functions 
  and changing the
    sign of $\sum_{i=0}^{\ell-2} x^{2i+1} c_{2i+1}(H)$;
  \item if $k = 2\ell$ is even and $H = \xi^2 + x^{2\ell}$, then the odd-indexed functions $c_{2i+1}, 0 \le i \le \ell-2,$
    are uniquely defined up to addition of flat functions and changing the sign of
    $\sum_{i=0}^{\ell-2} x^{2i+1} c_{2i+1}(H)$.
  \end{itemize}
\end{theo}
\begin{proof}
  The proof is completely parallel to that of
  Theorem~\ref{theo:classical-smooth}.
\end{proof}
Using this normal form, one can deduce a similar fibration-preserving
classification. Specifically, we have the following result.
\begin{theo}
  \label{theo:classical-smooth_fibration_preserving_case}
  Let $H$ be a $C^\infty$-smooth Hamiltonian on $(\RM^2, \omega)$
having $A_{k-1}, k \ge 2,$ singularity at the origin $O$.
  Then under the fibration-preserving equivalence, the pair $(H, \omega)$ has the following local symplectic $C^\infty$-smooth normal form:
  \begin{equation}
    \big(H = \xi^2 \pm x^{k}, \ \omega = (1+\sum_{i=1}^{k-2} x^i c_i(H)) \textup{d} \xi
  \wedge \textup{d} x \big)
  \end{equation}
  for some $C^\infty$-smooth  functions $c_i$  defined  uniquely modulo the following relations:
  \begin{itemize}
  \item if $k = 2\ell + 1$ is odd, then $c_{i}$'s are uniquely defined up to addition of flat functions (that is,
  the Taylor series of $c_{i}$'s form a complete and minimal set of symplectic invariants);
  \item if $k = 2\ell$ is even and $H = \xi^2 - x^{2\ell}$, then $c_{i}$'s are uniquely defined up to addition of flat functions 
  and changing the
    sign of $\sum_{i=0}^{\ell-2} x^{2i+1} c_{2i+1}(H)$;
  \item if $k = 2\ell$ is even and $H = \xi^2 + x^{2\ell}$, then the odd-indexed functions $c_{2i+1}, 0 \le i \le \ell-2,$
    are uniquely defined up to addition of flat functions and changing the sign of
    $\sum_{i=0}^{\ell-2} x^{2i+1} c_{2i+1}(H)$.
  \end{itemize}
\end{theo}
\begin{proof}
First consider the case when the leaves of the foliation given by $H$ are connected, that is,
when $k$ is odd or $k = 2\ell$ is even and $H = \xi^2 + x^{2\ell}.$  Let
$\tilde H$ be a Hamiltonian  defining the same (singular) foliation as $H.$ Then
$\tilde H = Hg(H)$ for a $C^\infty$-smooth germ $g$ with $g(0) \ne 0.$ Without loss of generality,
we can assume that $g(0) > 0,$ since otherwise there exists no diffeomorphism $\psi$ such that 
$H = \tilde H \circ \psi.$ Now observe 
that there is a unique function $g(H)$ with $g(0) > 0$ such that
$c_0(Hg(H)) = 1$. Indeed, the same proof as the one given in
Theorem~\ref{theo:classical-analytic_fibration_preserving_case} works
equally well in the $C^\infty$-smooth category.

Now consider the case when the leaves of the foliation 
induced by $H$ are not connected, that is, when $H = \xi^2 - x^{2\ell}.$ Then the same proof allows us to simplify the
$H$-preserving normal form to 
\begin{equation}
\label{eq:normal-2l}
  \big(H = \xi^2 - x^{2 \ell}, \ \omega = (1+\sum_{i=1}^{2\ell-2} x^i c_i(H)) \textup{d} \xi
  \wedge \textup{d} x \big),
  \end{equation}
  where $c_{i}$'s can be modified by arbitrary flat functions. Since
  the leaves of the level sets $H = h$ are not connected in this case,
  it is a priori possible to simplify $c_{i}$'s further by allowing
  Hamiltonians $\tilde H$ that define the same (singular) foliation as
  $H$, but cannot be written in the form $\tilde H = H g(H)$ for a
  $C^\infty$-smooth germ $g = g(H)$.  However, we claim that this does
  not happen. Indeed, the formula $\tilde H = H g(H)$ still holds {\em
    formally} at the level of Taylor series: for every
  $C^\infty$-smooth Hamiltonian $\tilde H$ defining the same
  (singular) foliation as $H$, there exists a $C^\infty$-smooth germ
  $g = g(H)$ with $g(0) \ne 0$ such that the Taylor series of
  $\tilde H$ and $H g(H)$ at the origin coincide. In view
  of~\eqref{equ:hat_c_i}, this means that the Taylor series of
  $c_{i}$'s remain invariant.  The result follows.
\end{proof}

\section{Discussion}
\label{sec/discussion}

In this work we have given a complete symplectic classification of
$A_k$ singularities in the $C^\infty$-smooth and real-analytic
categories. We have showed in particular that for
$A_{2\ell -1}, \ell > 1,$ singularities given by a Hamiltonian
function $H = \xi^2 + x^{2\ell}$ the list of symplectic invariants is
given by a number of function germs and Taylor series invariants,
where the Taylor series invariants are `invisible' to the action
variables. In all other cases the symplectic invariants are given by
Taylor series in both $C^\infty$-smooth and analytic categories, and
the action variables are sufficient for the symplectic classification.
One advantage of our proofs of these results is that the normal form derivation
is explicit and can be applied to all specific integrable models with $A_k$ singularities,
in a similar way this is done in \cite{Francoise2013} for the non-degenerate case.
Note however that we do not make any claims as to the computational
effectiveness of our algorithm.

\begin{figure}[htbp]
  \begin{center}
    \includegraphics[width=0.62\linewidth]{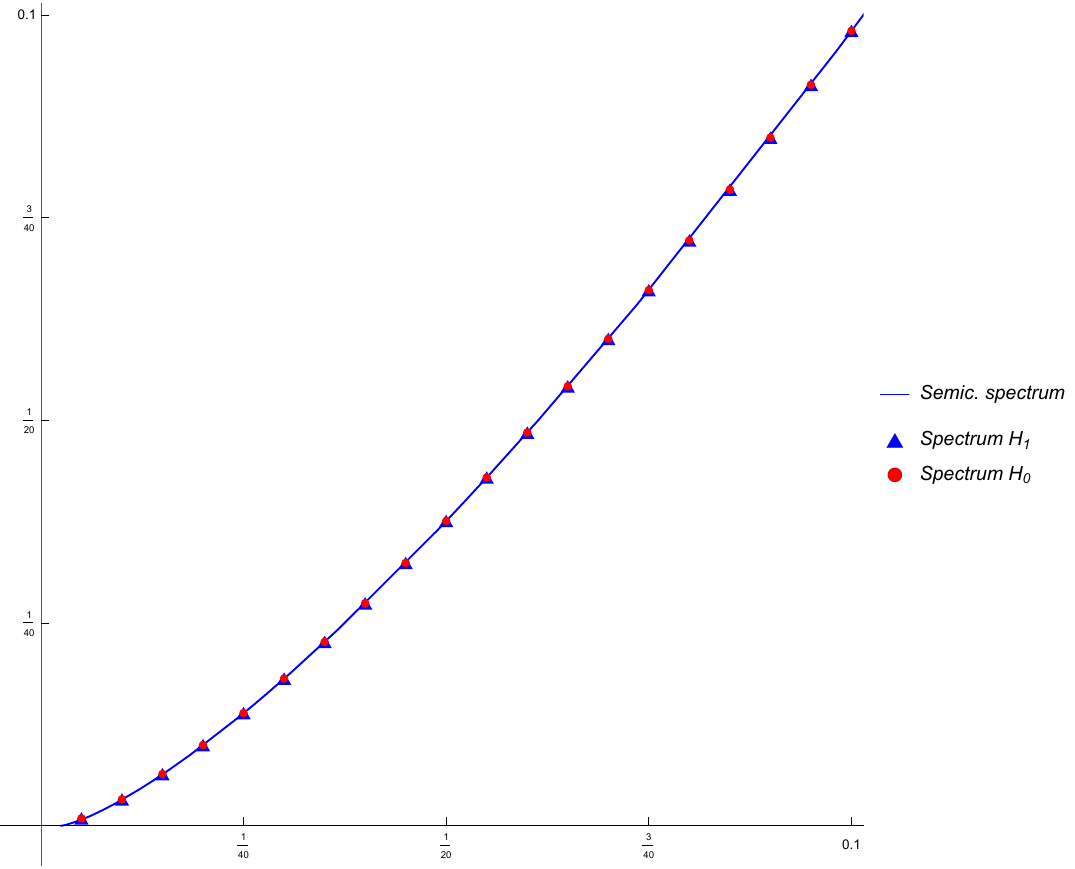}
    \caption{Spectra of the Hamiltonians $\widehat{H}_0 = - \hbar^2 \frac{\partial^2}{\partial x^2} + x^4$ (red dots) and 
    $\widehat{H}_1 = - \hbar^2 \frac{\partial^2}{\partial x^2} + V_1(x)$ (blue triangles)
    and their semi-classical approximation (blue curve). The constant $c = \frac{1}{16}$ and $\hbar = \frac{1}{200}$ in this case.}
    \label{Spectrum_A4}
  \end{center}
\end{figure}

Finally, we would like to note that the classification makes the compact $A_{2\ell -1}, \ell > 1,$ singularities
especially interesting from the view-point of the inverse spectral
problem of integrable systems. At present, we do not know if such
singularities are spectrally determined (the case of other $A_k$
singularities is simpler since then the action variables alone are the
only symplectic invariants). As an example, consider again
two Hamiltonian systems
$$
\big(H = \xi^2 + x^{2\ell},\, \omega_0 = \dd \xi \wedge \dd x\big)
\quad \mbox{ and } \quad \big(H = \xi^2 + x^{2\ell},\, \omega_1 =
(1+xc_1(x^{2\ell}))\dd \xi \wedge \dd x\big),
$$
where $\ell \ge 2$ and $c_1$ is a smooth function whose Taylor
series at the origin does not vanish identically. For simplicity, let $\ell = 2$ and $c_1(x^4) \equiv c.$
Performing the 
change of variables $\tilde x = x+ \frac{c x^2}{2}$ 
brings the symplectic structure $\omega_1$ to the canonical form, while the potential
becomes 
$V_1(\tilde x) = x(\tilde x)^{4} = \frac{(-1+\sqrt{1 + 2 c \tilde x})^4}{c^4}.
$
We thus get another pair of Hamiltonian systems 
$$
(H_0 = \xi^2 + x^{2\ell},\, \dd \xi \wedge \dd x) \quad \mbox{ and }
\quad (H_1 = \xi^2 + V_1(x),\, \dd \xi \wedge \dd {x}),
$$
where the potential 
\begin{equation} \label{potential_V1} 
V_1(x) = \frac{(-1+\sqrt{1 + 2 c x})^4}{c^4}.
\end{equation}
By the construction, these systems have identical action variables (whenever the energy $h < \frac{1}{c^4}$),
but are not fiberwise symplectomorphic near the origin. By the Bohr-Sommerfeld rules, the spectra of the corresponding quantised Hamiltonians 
$$
\widehat{H}_0 = - \hbar^2 \frac{\partial^2}{\partial x^2} + x^4 \quad \mbox{ and } \quad \widehat{H}_1 = - \hbar^2 \frac{\partial^2}{\partial x^2} + V_1(x),
$$
where $V_1$ is as in \eqref{potential_V1}, differ by a quantity of size $O(\hbar^2)$ (for small energies).
However, these spectra do not coincide; see Figure~\ref{Spectrum_A4}. Even more, one can show that the knowledge of the
spectrum of a Hamiltonian $\widehat H = - \hbar^2 \frac{\partial^2}{\partial x^2} + V(x)$ up to $O(\hbar^3)$ allows one to completely reconstruct the potential $V$ provided
that it is analytic or satisfies a mild regularity assumption; see \cite{CdV2011}. The implication of this
result is that
the
spectrum of a quantum Hamiltonian with an $A_k$ singularity at the origin contains more information than the action variables alone.
We conjecture that it allows one to reconstruct all of the symplectic invariants constructed in Sections~\ref{sec/classification_analytic} and
\ref{sec/classification_smooth} (that is, the function $c_1 = c_1(x^{2\ell})$ in the above example).

\appendix

\section{Proof of Theorem~\ref{theorem/remaining_cases}}

In this section, we prove Theorem~\ref{theorem/remaining_cases}.  For
convenience, let us recall the statement of this theorem here.

\begin{theo} 
  Let $H = \xi^2 - x^k$ with $k \ge 2$ and $\omega_0$ and $\omega_1$
  be arbitrary symplectic forms near $0\in\RM^2$.  Assume that the
  corresponding local action variables are equal up to a
  $C^\infty$-smooth function $f$:
$$
I_0(h) := \int\limits_{R(h)} \omega_0 = I_1(h) - f(h) :=
\int\limits_{R(h)} \omega_1 - f(h),
$$
where the region $R(h)$ is enclosed by the level curves $H = 0,$
$x = (\xi^2 - h)^{1/k}, \ h \le 0,$ and $\xi = \pm \varepsilon;$ see
Figure~\ref{Action_vars}.  Then there exists a local $C^\infty$-smooth
$H$-preserving isotopy $\psi_t$ of Moser type, such that
\[
  \psi_1^* \omega_1 = \omega_0\,.
\]
\end{theo}
\begin{proof}

  \mbox{ }

  \noindent\textbf{Step 1.} Write $\omega_1 - \omega_0$ as
  $\omega_1 - \omega_0 = g(x,\xi) \dd \xi \wedge \dd x$ for some
  smooth germ $g = g(x,\xi)$. By assumption,
  \[
    I(h) = \int\limits_{R(h)} g(x,\xi) \dd \xi \wedge \dd x = f(h)
  \]
  and hence
  \[
    f'(h) = \int_{-\varepsilon}^{\varepsilon} \frac{g((\xi^2 -
      h)^{\frac{1}{k}},\xi)}{k (\xi^2 - h)^{\frac{k-1}{k}}} \dd \xi
  \]
  are $C^\infty$-smooth on $[-h_0, 0]$ with $h_0 > 0$ sufficiently
  small. We shall deduce from this that
  \begin{equation} \label{eq/cohomological} g(x,\xi) \dd \xi \wedge
    \dd x = \dd H \wedge \dd u,
  \end{equation} for some smooth $u$, which will give the result
  via Moser's trick (see Lemma~\ref{lemm:moser}). 

  Observe that on the set $R(h_0)$, our cohomological equation
  $g(x) \dd \xi \wedge \dd x = \dd H \wedge \dd u$ admits a solution
  \begin{equation} \label{eq/defn_u} u(x,\xi) = \tilde{u} (\xi, \xi^2
    - x^k), \quad \mbox{ where } \quad \tilde u(\xi,h) = \int^\xi_{-\varepsilon}
    \dfrac{g((t^2-h)^{1/k},t)}{k(t^2-h)^{\frac{k-1}{k}}} dt.
  \end{equation}
  Let us first show that the thus defined $u$ is $C^\infty$-smooth
  (more precisely, admits a $C^\infty$-smooth extension to a small
  open neighbourhood containing the set $R(h_0)$).

 \noindent\textbf{Step 2.} 
 To show that $u$ defined in \eqref{eq/defn_u} is $C^\infty$-smooth,
 we first observe that
$$g = \sum_{i}^{n-1} d_i x^{i} + R(x,\xi^2)x^{n} + \dd v \wedge \dd H,$$
where $v = v(x,\xi)$ is smooth and $d_i = 0$ when
$i \equiv k-1 \mod k$. Indeed, first observe that without loss of
generality, $g = g(x, \xi^2)$ is an even function of $\xi$, for
formula \eqref{eq/solution_even} works equally well in the smooth
case.  Next consider the Taylor expansion
\begin{equation} \label{eq/Taylor_expansion_1} g(x,\xi^2) = \sum_{i,j
    = 0}^n A_{ij} x^i \xi^{2j} + \xi^{2n+2} \sum_{i = 0}^n x^i
  r_i(\xi^2) + x^{n+1} R_1(x,\xi^2),
\end{equation} 
where the integer $n$ is arbitrary large and the functions $r_k,$ and
$R_1$ are $C^\infty$-smooth functions of their variables. Exactly as
in the analytic case discussed in Section~\ref{sec/classification_analytic}, summing up
monomials $\frac{1}{2 (k+1) }x^{k+1} \xi^{2j-1}$ with appropriate
coefficients, we can get rid of the variable $\xi$ from the first sum
of the expansion \eqref{eq/Taylor_expansion_1}.  Therefore,
\begin{equation}
  g(x,\xi^2) = \sum_{i = 0}^n d_{i} x^i + 
  \xi^{2n+2}  \sum_{i = 0}^n x^i r_i(\xi^2)  + x^{n+1} R_2(x,\xi^2) + \textup{d}u_2 \wedge \textup{d}H.
\end{equation}
for some smooth functions $R_2$ and $u_2$.  Similarly, an
appropriately chosen sum of functions of the form
$v_{kn} = x^{k+1} \xi^{2n+1} q_k(\xi^2),$ gives us that
\begin{equation}
  g(x,\xi^2) = \sum_{i = 0}^n d_{i} x^i + x^{n+1} R_3(x,\xi^2) + \textup{d}u_3 \wedge \textup{d}H
\end{equation}
for some smooth functions $R_3$ and $u_3$ as required. Finally, using
the explicit formula
\begin{equation}
  u_4 = \tilde{u_4}(\xi,H(x,\xi)), \quad \mbox{where } \quad  \tilde{u_4}(\xi,h) = \frac{1}{k} \int_{0}^\xi \tilde g_k(h-s^2, s) \textup{d}s,
\end{equation}
we can subtract an arbitrary smooth function of the form
$x^{k-1} g_k(x^k, \xi)$ from $g$.  In particular, we can achieve that
$d_i = 0$ for $i \equiv k-1 \mod k.$ Thus, without loss of generality,
\begin{equation} \label{eq/normal_g} g = \sum_{i}^{n-1} d_i x^{i} +
  R(x,\xi)x^{n}, \quad \mbox{ where } \ d_i = 0 \ \mbox{ for } \ i
  \equiv k-1 \mod k.
\end{equation}
\noindent\textbf{Step 3.}
Because of Eq.~\eqref{eq/normal_g}, we can now write
$\tilde{u}(\xi,h)$ as
\[
  k \tilde{u}(\xi,h) = \int^\xi_{-\varepsilon} \sum_i d_i
  (t^2-h)^{\frac{i+1-k}{k}} dt+ \int^\xi_{-\varepsilon}
  R((t^2-h)^{1/k},t^2)(t^2-h)^{\frac{n+1-k}{k}} dt.
\]
Observe that we can write the derivative $k f'(h)$ as
\[
  k f'(h) = \int^{\varepsilon} _{-\varepsilon} \sum_i d_i
  (t^2-h)^{\frac{i+1-k}{k}} dt+ \int^{\varepsilon} _{-\varepsilon}
  R((t^2-h)^{1/k},t^2)(t^2-h)^{\frac{n+1-k}{k}} dt,
\]
which is $C^\infty$-smooth by the assumption of the theorem. But the
integral
$$
\int^{\varepsilon} _{-\varepsilon} (t^2-h)^{\frac{i+1-k}{k}} dt
$$
is not a smooth function of $h$ for $i \not\equiv k-1 \mod k$ since
its $h$-expansion has the form
$$
\int^{\varepsilon} _{-\varepsilon} (t^2-h)^{\frac{i+1-k}{k}} dt =
c(i,k) (-h)^{\frac{i+1-k}{k} + 1/2} + s(h),
$$
where $s$ is real-analytic and $c(i,k)$ is a constant which is
non-zero for $i \not\equiv k-1 \mod k$.

[Indeed, we can write
$$\int^{\varepsilon} _{-\varepsilon}  (t^2-h)^{\frac{i+1-k}{k}} dt = (-h)^{\frac{i+1-k}{k} + 1/2} \int^{\varepsilon/\sqrt{-h} } _{-\varepsilon/\sqrt{-h} }  (1+s^2)^{\frac{i+1-k}{k}} ds.$$
Differentiating the integral
$ \int^{\varepsilon/\sqrt{-h} }_{-\varepsilon/\sqrt{-h} }
(1+s^2)^{\frac{i+1-k}{k}} ds $ with respect to $h$ and then
integrating yields
$$
\int^{\varepsilon/\sqrt{-h} } _{-\varepsilon/\sqrt{-h} }
(1+s^2)^{\frac{i+1-k}{k}} ds = (-h)^{\frac{k-i-1}{k} - 1/2} s(h) +
c(i,k),
$$
where $s$ is $C^\infty$-smooth (even real-analytic) and $c(i,k)$ is a
constant. Hence
$$\int^{\varepsilon} _{-\varepsilon}  (t^2-h)^{\frac{i+1-k}{k}} dt = c(i,k) (-h)^{\frac{i+1-k}{k} + 1/2} + s(h).$$
To show that the coefficient $c(i,k) \ne 0$ for
$i \not\equiv k-1 \mod k$, it suffices to differentiate the expression
$$(-h)^{\frac{i+1-k}{k} + 1/2} \int^{\varepsilon/\sqrt{-h} } _{-\varepsilon/\sqrt{-h} }  (1+s^2)^{\frac{i+1-k}{k}} ds$$
with respect to $h$: If $c(i,k)$ is zero, then it must be smooth. On
the other hand, the derivative is then
$$\varepsilon  (-h)^{\frac{i+1-k}{k} - 1} (1-\varepsilon^2/h)^{\frac{i+1-k}{k}} - \frac{i+1-k/2}{k} \frac{s(h)}{h},$$
which is not $C^\infty$-smooth for $i \not\equiv k-1 \mod k$. This
contradiction shows $c(i,k) \ne 0$.]

This proves that all the coefficients $d_i = 0$ when
$[\frac{i+1}{k}] < [\frac{n+1}{k}]$, so that $\tilde{u}(\xi,h)$ is
continuously differentiable at least $[\frac{n+1}{k}]-1$ times. Hence
$u$ is also continuously differentiable at least $[\frac{n+1}{k}]-1$
times. Letting $n \to \infty$, we conclude that $u$ is
$C^\infty$-smooth.

\noindent\textbf{Step 4.} In case when $k$ is odd, the proof of the theorem is complete, since then the function $u$ given by
\eqref{eq/defn_u} is well defined for both positive and negative
values of $h$ (exactly the same argument as in Step 3 shows that $u$
is $C^\infty$-smooth for all $h$ sufficiently close to zero).
 
In case when $k$ is even, Step 3 gives a $C^\infty$-smooth function
$u$, which can by virtue of the formulas
 $$
 u = \tilde u(\xi,H) \quad \mbox{and} \quad \tilde u(\xi,h) =
 \int^\xi_{-\varepsilon}
 \dfrac{g((t^2-h)^{1/k},t)}{k(t^2-h)^{\frac{k-1}{k}}} dt =
 \int^\xi_{-\varepsilon} \dfrac{g(|t^2-h|^{1/k},t)}{k
   |t^2-h|^{\frac{k-1}{k}}} dt, \quad h \le 0,
 $$
 be extended to a small open neighbourhood of $R(h_0)$. Moreover, the
 function $u$ solves the cohomological equation
 \eqref{eq/cohomological} on the set $R(h_0)$.  To obtain a global
 solution to the cohomological equation \eqref{eq/cohomological}, we
 extend the definition of $u$ from the domain $R(h_0)$ to a full
 tubular neighbourhood of the singular fiber $H = \xi^2 - x^k = 0$ as
 follows: Setting $\xi = \pm \varepsilon, \ \varepsilon > 0,$ gives
 four Lagrangian sections of the fibration given by $H$; see
 Figure~\ref{Quartic_minus}.
   \begin{figure}[htbp]
   \begin{center}
     \includegraphics[width=0.48\linewidth]{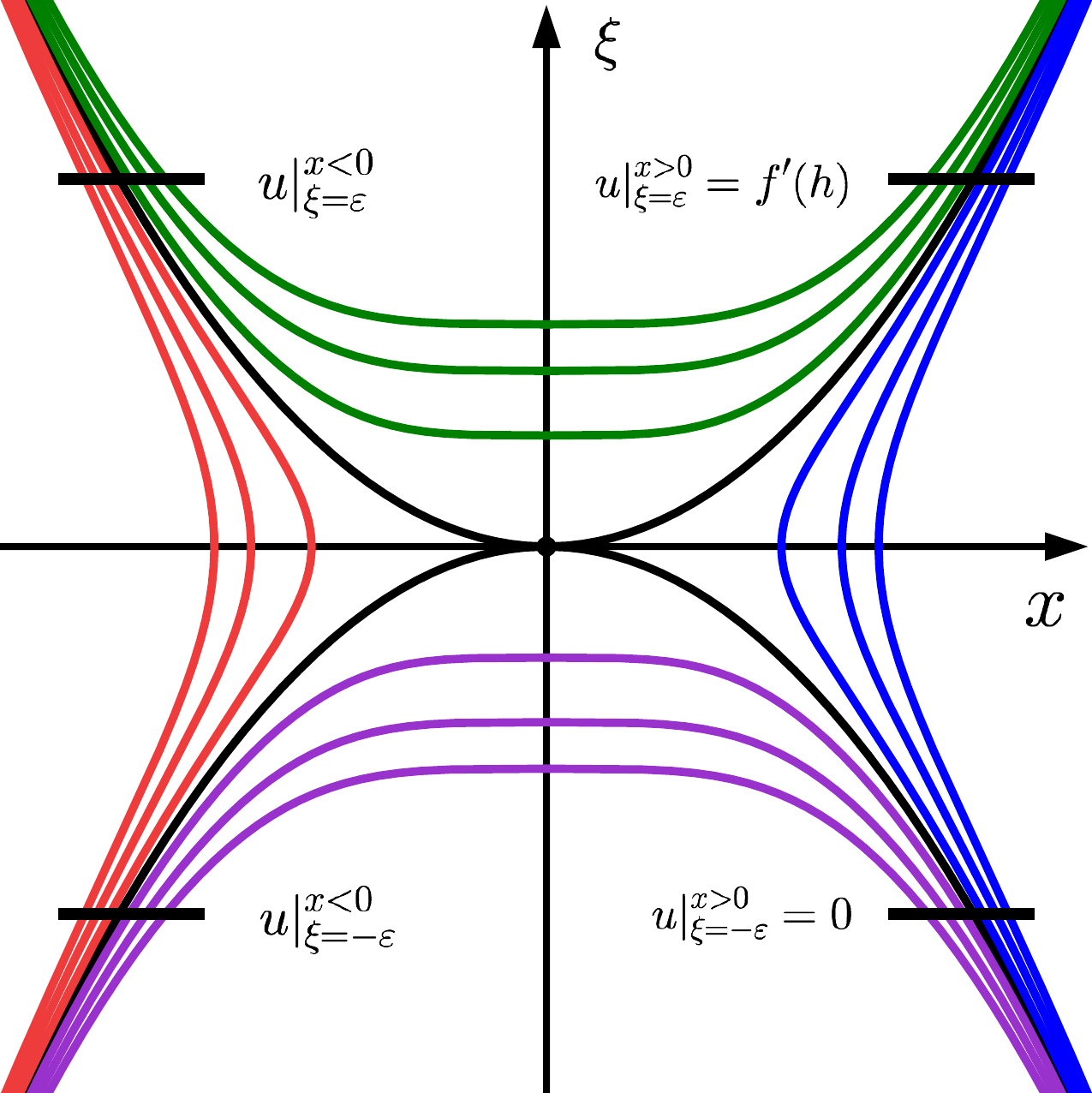}
     \caption{Level sets of $A_{3}$ singularity $H = \xi^2 -
       x^{4}$. The four Lagrangian sections $\xi = \pm \varepsilon$,
       $x \lessgtr 0$ are shown in black. The four parts of the
       fibration, corresponding to the four functions $v_i,$ are shown
       in colour. }
     \label{Quartic_minus}
   \end{center}
 \end{figure}
 Integrating $g(x,\xi)$ along the Hamiltonian flow of
 $$\dot{x} = H_\xi, \quad \dot \xi = - H_x$$
 between these Lagrangian sections (i.e. along the characteristics of
 the linear PDE $\{H, u\} = g(x,\xi)$), produces 4 functions $v_i$
 given explicitly as follows:
 \begin{align}
   v_1(\xi, h) \equiv \tilde u(\xi, h) = \int_{-\varepsilon}^{\xi} \frac{g((t^2 - h)^{\frac{1}{k}},t)}{-k (t^2 - h)^{\frac{k-1}{k}}} \dd t, \quad    v_3(\xi, h) = \int_{\varepsilon}^{\xi} 
   \frac{g(-(t^2 - h)^{\frac{1}{k}},t)}{k (t^2 - h)^{\frac{k-1}{k}}} \dd t
 \end{align}
 and
 \begin{align}
   v_2(x, h) = \int_{(\varepsilon^2-h)^{1/k}}^{x} \frac{g(s, \sqrt{h+s^k})}{2 \sqrt{h+s^k}} \dd s, 
   \quad 
   v_4(x,h) = \int_{-(\varepsilon^2-h)^{1/k}}^{x} \frac{g(s, -\sqrt{h+s^k})}{-2 \sqrt{h+s^k}} \dd s.
 \end{align}
 The global solution $u$ is now defined by the following rule:
 \begin{itemize}
 \item[$\bullet$] If $(x,\xi)$ is such that $x \ge 0$ and
   $H(x,\xi) \le 0$, then $u(x,\xi) = \tilde u(\xi, H) = v_1(\xi, H) $
   as above;
 \item[$\bullet$] If $(x,\xi)$ is such that $\xi \ge 0$ and
   $H(x,\xi) \ge 0$, we set $u(x,\xi) = h_1(H) + v_2(x, H),$ where
   $h_1(H)$ is any $C^\infty$-smooth function that coincides with
   $v_1(\varepsilon, H) \equiv f'(H)$ on the Lagrangian
   section $\xi = \varepsilon, x > 0$;
 \item[$\bullet$] Next, if $(x,\xi)$ is such that $x \le 0$ and
   $H(x,\xi) \le 0$, we set
   $u(x,\xi) = h_1(H) + h_2(H) + v_3(\xi, H),$ where $h_2(H)$ is any
   $C^\infty$-smooth function  that coincides with
   $v_2(-(\varepsilon^2-H)^{1/k}, H)$ on the Lagrangian section
   $\xi = \varepsilon, x < 0$;
 \item[$\bullet$] Finally, if $(x,\xi)$ is such that $\xi \le 0$ and
   $H(x,\xi) \ge 0$, we set
   $u(x,\xi) = h_1(H) + h_2(H) + h_3(H) + v_4(x, H),$ where $h_3(H)$
   is any $C^\infty$-smooth function that coincides with
   $v_3(-\varepsilon, H)$ on the Lagrangian section
   $\xi = - \varepsilon, x < 0$.
 \end{itemize}
 Observe that by the construction, we thus have $h_1(H) \equiv f'(H)$
 when $H \le 0$. Moreover,
$$h_1(0) + h_2(0) + h_3(0) + h_4(0) = 0, \quad \mbox{where} \quad h_4(0) = v_4(\varepsilon^{2/k}, 0),$$ 
since in this case we integrate twice along each of the 4 branches of
$\xi^2 - x^{2\ell} = 0, \, k = 2 \ell,$ with opposite
orientations. This shows that $u = u(x,\xi)$ is a well-defined and
continuous function in a neighbourhood of the origin.  Proceeding
exactly as in Step 3, one can show that the functions $v_1(\xi, H)$
and $v_3(\xi, H)$ and hence also $h_3(H)$ are
$C^\infty$-smooth. Writing
\begin{align}
  v_2(x(h,\xi), h) = \int_{\varepsilon}^{\sqrt{h}} \frac{g((t^2 - h)^{\frac{1}{k}},t)}{k (t^2 - h)^{\frac{k-1}{k}}} \dd t +
  \int_{\sqrt{h}}^{\xi} 
  \frac{g((t^2 - h)^{\frac{1}{k}},t)}{k (t^2 - h)^{\frac{k-1}{k}}} \dd t,
\end{align}
we see that also $v_2(\xi, H)$ are hence $h_2(H)$ are
$C^\infty$-smooth. Indeed, we have already shown that in the
preliminary normal form \eqref{eq/normal_g}, all of the coefficients
$d_i = 0$ when $[\frac{i+1}{k}] < [\frac{n+1}{k}]$, which shows that
these functions are at least $[\frac{n+1}{k}]-1$ times continuously
differentiable.  But the exponent $n$ is arbitrary large, which shows
that these functions are in fact $C^\infty$-smooth.

A similar argument shows the smoothness of $v_4(\xi, H)$.  We therefore get that $u$, being a sum of $v_i$ and $h_j$ on the closures of each
connected component of $H^{-1}([-h_0, h_0]) \setminus H^{-1}(0)$, is
$C^\infty$-smooth on each of these individual closures. 
Moreover, the vanishing of the coefficients
$d_i = 0$ in \eqref{eq/normal_g} implies that the corresponding derivatives  match on 
$H^{-1}(0).$ It follows that $u$ is
$C^\infty$-smooth.
\end{proof}

\bibliographystyle{abbrv}
\bibliography{./library}
\end{document}